%% file: Kims-lemma-for-bi-invariant-types.tex
\documentclass[English]{amsart}%

\usepackage[utf8]{inputenc}
\usepackage{amsmath} %
\usepackage{amsthm}
\usepackage{amsfonts}
\usepackage{amssymb}

\usepackage[figuresright]{rotating}
\usepackage{fullpage}
\usepackage{mathtools}
\usepackage{stmaryrd}
\usepackage{mathrsfs}
\usepackage[bookmarks,hidelinks,pdfencoding=unicode]{hyperref}

\usepackage[capitalize]{cleveref}
\usepackage{tikz}
\usetikzlibrary{calc}

\usepackage{listings}
\usepackage{enumerate}
\usepackage{bbm}
\usepackage{accents}
\usepackage{nicefrac}
\usepackage{quiver}

\usepackage{microtype}

\usepackage{gensymb}

\input{macros.tex}

\newenvironment{amssidewaysfigure}
  {\begin{sidewaysfigure}\vspace*{.5\textwidth}\begin{minipage}{\textheight}\centering}
  {\end{minipage}\end{sidewaysfigure}}

\renewcommand{\top}{\mathrm{top}}

\newcommand{\eval}{\mathrm{eval}}

\newcommand{\graphjoin}{\mathop{\nabla}}

\newcommand{\KL}{-Kim's lemma}%

\newcommand{\Weav}{\mathsf{Weav}}

\newcommand{\Elem}{\mathsf{Elem}}

\makeatletter
\newcommand\footnoteref[1]{\protected@xdef\@thefnmark{\ref{#1}}\@footnotemark}
\newcommand\doublefootnoteref[2]{\protected@xdef\@thefnmark{\ref{#1},\ref{#2}}\@footnotemark}
\makeatother

\allowdisplaybreaks

\begin{document}

\title{A combinatorial characterization of Kim's lemma for pairs of bi-invariant types}

\address{Department of Mathematics \\
  Iowa State University \\
  396 Carver Hall \\
  411 Morrill Road \\
  Ames, IA 50011, USA}
\author{James E. Hanson}
\email{jameseh@iastate.edu}
\date{\today}


\begin{abstract}
  We give a combinatorial consistency-inconsistency configuration that is equivalent to the failure of the following form of Kim's lemma for a given $k$:
  \begin{itemize}
  \item[$(\star)$] For any set of parameters $A$, formula $\varphi(x,b)$, and $A$-bi-invariant types $p$ and $q$ extending $\tp(b/A)$, if $\varphi(x,b)$ $k$-divides along $p$, then it divides along $q$.
  \end{itemize}
   We then give an equivalent technical variant of $(\star)$ that is non-trivial over arbitrary invariance bases. We also show that the failure of weaker versions of $(\star)$ entails the existence of stronger combinatorial configurations, the strongest of which can be phrased in terms of families of parameters indexed by arbitrary cographs (i.e., $P_4$-free graphs).

  Finally, we show that if there is an array $(b_{i,j} : i,j < \omega)$ of parameters such that $\{\varphi(x,b_{i,j}) : (i,j) \in C\}$ is consistent whenever $C \subseteq \omega^2$ is a chain (in the product partial order) and $k$-inconsistent whenever $C$ is an antichain, then there is a model $M$, parameter $b$, and $M$-coheirs $p,q \supset \tp(b/M)$ such that $q^{\otimes \omega}$ is an $M$-heir-coheir and $\varphi(x,b)$ $k$-divides along $p$ but does not divide along $q$. In doing so, we also show that this configuration entails the failure of generic stationary local character under the assumption of \GCH.

\end{abstract}

\maketitle

\section*{Introduction}
\label{sec:intro}

This paper is a direct continuation of \cite{NCTP}, which studied the comb tree property or CTP (originally introduced by Mutchnik as $\omega$-DCTP$_2$ in \cite{Mutchnik-NSOP2}). The negation of CTP, NCTP, is one of three studied mutual generalizations of NTP$_2$ and NSOP$_1$, the other two being (the negation of) the antichain tree property or NATP, introduced by Ahn and Kim in \cite{ATP-1}, and (the negation of) the bizarre tree property or NBTP, introduced by Kruckman and Ramsey in \cite{NKL}. An important aspect of a lot of this work is the behavior of certain classes of special invariant types, which feature prominently in \cite{NCTP} and in this paper.

\begin{defn} Recall the following anchors:
  \begin{itemize}
  \item $c \indf_A b$ means that $\tp(c/Ab)$ does not fork over $A$.
  \item $c \indK_A b$ means that $\tp(c/Ab)$ does not Kim-fork over $A$.
  \item $c \indi_A b$ means that $\tp(c/Ab)$ extends to an $A$-invariant type.
  \item $c \indu_A b$ mean that $\tp(c/Ab)$ is finitely satisfiable in $A$.
  \end{itemize}
  Fix an $A$-invariant type $p(x)$.
  \begin{itemize}
  \item $p(x)$ is \emph{strictly $A$-invariant} if whenever $b \models p \res A c$, $c \indf_A b$.
  \item $p(x)$ is \emph{Kim-strictly $A$-invariant} if whenever $b \models p \res A c$, $c \indK_A b$.
  \item $p(x)$ is \emph{$A$-bi-invariant} if whenever $b \models p \res A c$, $c \indi_A b$.
  \item $p(x)$ is \emph{$n$-strongly $A$-bi-invariant}\footnote{\label{foot:new}These definitions are new, although the concept of a strong heir-coheir is implicit in \cite[Fact~0.4]{NCTP}.} if $p^{\otimes n}$ is $A$-bi-invariant.
  \item $p(x)$ is \emph{strongly $A$-bi-invariant} if it is $\omega$-strongly $A$-bi-invariant.
  \item $p(x)$ is an \emph{$A$-heir-coheir} if $p(x)$ is an $A$-coheir and whenever $b \models p \res A c$, $c \indu_A b$.
  \item $p(x)$ is an \emph{$n$-strong $A$-heir-coheir}\footnoteref{foot:new} if $p^{\otimes n}$ is an $A$-heir-coheir.
  \item $p(x)$ is a \emph{strong $A$-heir-coheir}\footnoteref{foot:new} if is an $\omega$-strong $A$-heir-coheir.
  \item $p(x)$ is \emph{extendibly $A$-invariant} if for any type $q(x,\ybar)$ extending $p\res A$, $p(x) \cup q(x,\ybar)$ extends to an $A$-invariant type.
  \end{itemize}
  
\end{defn}

\cite{NCTP} also introduced a technical strengthening of Kim-strict invariance called \emph{reliable invariance} and an intermediate notion of \emph{semi-reliable invariance} (\cref{defn:semi-reliability}), which will play a role in this paper. Figure~\ref{fig:invariant} contains implications between these notions known to the author.

\begin{figure}
  \centering
\[\begin{tikzcd}
	{\text{Generically stable (fim)\footnoteref{foot:wild}}} \\
	{\text{Finitely approximated (fam)\footnoteref{foot:wild}}} & {\text{Definable coheir (dfs)\doublefootnoteref{foot:wild}{foot:definable}}} & {\text{Definable\footnoteref{foot:definable}}} \\
	& {\text{Strong heir-coheir}} & {\text{Strongly bi-invariant}} \\
	{\text{Canonical coheir\footnoteref{foot:canonical}}} & {n\text{-strong heir-coheir}} & {n\text{-strongly bi-invariant}} \\
	{\text{Reliable coheir}} & {\text{Heir-coheir}} & {\text{Bi-invariant}} \\
	{\text{Reliably invariant}} & {\text{Coheir}} \\
	{\text{Semi-reliably invariant}} & {\text{Extendibly invariant}} & {\text{Strictly invariant}} \\
	& {\text{Kim-strictly invariant}}
	\arrow[from=1-1, to=2-1]
	\arrow[from=2-1, to=2-2]
	\arrow[from=2-2, to=2-3]
	\arrow["{\text{over models}}", from=2-2, to=3-2]
	\arrow["{\text{over models}}", from=2-3, to=3-3]
	\arrow[from=3-2, to=3-3]
	\arrow[from=3-2, to=4-2]
	\arrow[from=3-3, to=4-3]
	\arrow[from=4-1, to=6-2]
	\arrow[from=4-2, to=4-3]
	\arrow[from=4-2, to=5-2]
	\arrow[from=4-3, to=5-3]
	\arrow[from=5-1, to=6-1]
	\arrow[from=5-1, to=6-2]
	\arrow[from=5-2, to=5-3]
	\arrow[from=5-2, to=6-2]
	\arrow[from=5-3, to=7-3]
	\arrow[from=6-1, to=7-1]
	\arrow[from=6-1, to=7-2]
	\arrow["{\text{over models}}", from=6-2, to=7-2]
	\arrow[from=7-1, to=8-2]
	\arrow[from=7-3, to=8-2]
\end{tikzcd}\]
  \caption{Some special classes of invariant types.}
  \label{fig:invariant}
\end{figure}
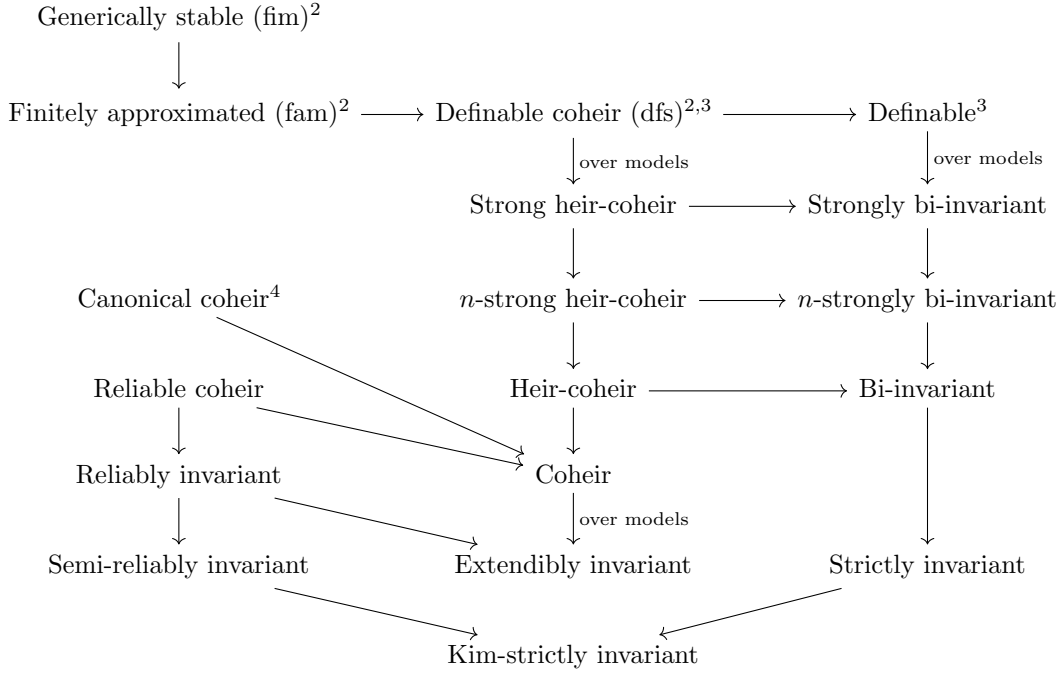
\addtocounter{footnote}{1}
\footnotetext{\label{foot:wild}See \cite{Conant2020,KeislerMeasuresWild} for an overview of generically stable and finitely approximated types as well as definable coheirs (also called dfs types).}
\addtocounter{footnote}{1}
\footnotetext{\label{foot:definable}To see that definable types (resp.\ definable coheirs) are strongly bi-invariant (resp.\ strong heir-coheirs), note that if $p(x)$ is a global $M$-definable type for a model $M$, then $p(x)$ is an heir of $p\res M$. Since $p^{\otimes n}$ is also $M$-definable for every $n$, any such type is strongly bi-invariant (resp.\ a strong heir-coheir).}
\addtocounter{footnote}{1}
\footnotetext{\label{foot:canonical}\emph{Canonical coheirs}, introduced in  \cite{Mutchnik-NSOP2}, are Kim-strictly invariant if Kim-dividing is defined in terms of coheirs rather than arbitrary invariant types or if the theory in question is NATP \cite[Rem.~5.4]{Some-Remarks-Kim-dividing-NATP}. Canonical coheirs should be closely related to the \emph{reliable coheirs} of \cite{NCTP}, as their constructions are both closely related to the broom lemma, but to the author's knowledge this has not been mapped out carefully.}

The driving philosophy of \cite{NKL} (which inspired the work in \cite{NCTP}) is that a reasonable approach to finding a mutual generalization of NTP$_2$ and NSOP$_1$ is to look at the variants of Kim's lemma that characterize these classes of theories:
\begin{itemize}
\item $T$ is NTP$_2$ if and only if whenever $\varphi(x,b)$ divides over a model $M$, it divides along any Morley sequence generated by a strictly invariant type extending $\tp(b/M)$.
\item $T$ is NSOP$_1$ if and only if whenever $\varphi(x,b)$ divides along some Morley sequence generated by an invariant type extending $\tp(b/M)$, it divides along all Morley sequences generated by invariant types extending $\tp(b/M)$.
\end{itemize}
Unlike with combinatorial consistency-inconsistency configurations, it is easy to see how to systematically generalize these statements by slotting in two classes of indiscernible sequences. Since we will be dealing with many such generalizations in this paper, we introduce systematic nomenclature for them in \cref{defn:Kim's-lemmas}.

BTP, CTP, and ATP\footnote{Listed in increasing order of strength.} all immediately fall out of failures of certain variants of Kim's lemma: If $T$ has a set of parameters $A$ and a formula $\varphi(x,b)$ that $k$-divides along some $A$-invariant type $p(y) \supset \tp(b/A)$ but not along some other $A$-invariant type $q(y) \supset \tp(b/A)$, then 
\begin{itemize}
\item if $q(y)$ is Kim-strictly $A$-invariant, then $T$ has $k$-BTP \cite[Thm.~5.2]{NKL},
\item if $A$ is a model, $p(y)$ is an $A$-coheir, and $q(y)$ is a canonical $A$-coheir, then $T$ has $k$-CTP \cite[Thm.~4.9]{Mutchnik-NSOP2},
\item if $A$ is an invariance base, $p(y)$ is extendibly $A$-invariant, and $q(y)$ is reliably $A$-invariant, then $T$ has $k$-CTP \cite[Prop.~2.6]{NCTP},
\item if $q(y)$ is $A$-bi-invariant, then $T$ has $k$-CTP \cite[Prop.~1.7]{NCTP}, and
\item if $q(y)$ is strongly $A$-bi-invariant, then $T$ has ATP \cite[Prop.~1.7]{NCTP}.\footnote{$k$-ATP is equivalent to $2$-ATP by \cite[Lem.~3.20]{Ahn2022}.}
\end{itemize}
There seem to be many such statements. We prove a three-parameter family of statements of this form in \cref{thm:failure-of-Kim-bi-invariant-to-weaves}, but to give a simpler example, the proof of \cite[Prop.~1.7]{NCTP} can be easily adapted to show the following.

\begin{defn}\label{defn:right-n-comb-tree}
  For any $n \leq \omega$, the set of \emph{right-$n$-combs} in $2^{<\omega}$ is the smallest set of subsets of $2^{<\omega}$ containing the singletons and satisfying that for any right-$n$-combs $A$ and $B$, if $\sigma$ is the greatest common initial segment of $A\cup B$, every element of $A$ extends $\sigma \concat 0$, every element of $B$ extends $\sigma \concat 1$, and $|A| \leq n$, then $A \cup B$ is a right-$n$-comb.

  A theory $T$ has the \emph{$(k,n)$-comb tree property} or \emph{$(k,n)$-CTP} if there is a formula $\varphi(x,y)$ and a tree $(b_\sigma : \sigma \in 2^{<\omega})$ such that for any right-$n$-comb $C \subseteq 2^{<\omega}$, $\{\varphi(x,b_\sigma) : \sigma \in C\}$ is consistent and for any path $P \subseteq 2^{<\omega}$, $\{\varphi(x,b_\sigma) : \sigma \in P\}$ is $k$-inconsistent.
\end{defn}

\begin{prop}\label{prop:k-n-CTP}
  If $T$ has a set of parameters $A$ and a formula $\varphi(x,b)$ such that $\varphi(x,b)$ $k$-divides along some $A$-invariant type $p(y) \supset \tp(b/A)$ but does not divide along some $n$-strongly $A$-bi-invariant type $q(y) \supset \tp(b/A)$, then $T$ has $(k,n)$-CTP. \qed
\end{prop}

This ostensibly gives a whole hierarchy of combinatorial configurations intermediate between CTP (which is $(k,1)$-CTP in the above terminology) and ATP (which is $(k,\omega)$-CTP in the above terminology), and prima facie none of these are equivalent.\footnote{Moreover it should be noted that at the moment there isn't even a known separation between NBTP and NPM$^{(2)}$, defined in \cite[Def.~6.1]{Bailetti2024}. See Figure~\ref{fig:implications}.} (We will not be studying $(k,n)$-CTP in this paper beyond the observation of \cref{prop:k-n-CTP}, although the first part of \cref{defn:right-n-comb-tree} is morally similar to \cref{defn:narrow-wide-comb}.)

Prior to this paper, the only known statement in the opposite direction as \cref{prop:k-n-CTP} (i.e., going from a combinatorial configuration to a failure of some variant of Kim's lemma) in the regime of mutual generalizations of NTP$_2$ and NSOP$_1$ was the following:
\begin{itemize}
\item If $T$ has $k$-CTP, then there is a model $M$, a formula $\varphi(x,b)$, an $M$-coheir $p(y) \supset \tp(b/M)$, and an $M$-heir-coheir $q(y) \supset \tp(b/M)$ such that $\varphi(x,b)$ $k$-divides along $p$ but does not divide along $q$ \cite[Prop.~1.5,~3.1]{NCTP}.
\end{itemize}

In this paper we define two families of combinatorial consistency-inconsistency configurations---namely \emph{$(k,m,n)$-weaves} (\cref{defn:weave}) and \emph{$k$-grids} (\cref{defn:k-grid})---and prove three new results (although two of the proofs are essentially identical) in the same direction as \cite[Prop.~1.5,3.1]{NCTP} (i.e., the opposite direction of \cref{prop:k-n-CTP}):
\begin{itemize}
\item If $T$ has a $(k,1,1)$-weave of depth $\omega$, then there is a model $M$, a formula $\varphi(x,b)$, and $M$-heir-coheirs $p(y),q(y) \supset \tp(b/M)$ such that $\varphi(x,b)$ $k$-divides along $p$ but does not divide along $q$ (\cref{prop:get-heir-coheir-heir-coheir-from-weave}). 
\item If $T$ has an infinite $k$-grid, then there is a model $M$, a formula $\varphi(x,b)$, an $M$-coheir $p(y) \supset \tp(b/M)$, and a strong $M$-heir-coheir $q(y)$ such that $\tp(b/M)$ $k$-divides along $p$ but does not divide along $q$ (\cref{thm:grid-theorem}).
\item If $T$ has an infinite $k$-grid, then there is a model $M$, a formula $\varphi(x,b)$, a strong $M$-heir-coheir $p(y) \supset \tp(b/M)$, and an $M$-coheir $q(y)$ such that $\tp(b/M)$ $k$-divides along $p$ but does not divide along $q$ (\cref{thm:grid-theorem}).
\end{itemize}

As we show in \cref{thm:failure-of-Kim-bi-invariant-to-weaves} (the aforementioned three-parameter family of statements), $(k,m,n)$-weaves are what naturally arise from an instance of a formula $k$-dividing along an $m$-strongly bi-invariant type but not some $n$-strongly bi-invariant type (and if $m=1$ or $n=1$, the type in question can be replaced with a semi-reliably invariant type),\footnote{The rationale for considering $(k,m,n)$-grids for arbitrary $m,n \leq \omega$ (rather than just $m,n \in \{1,\omega\}$) is that it adds essentially no extra technical complexity to the \emph{proof} of \cref{thm:failure-of-Kim-bi-invariant-to-weaves} (although admittedly it adds some notational and conceptual complexity to the \emph{statement} of the result), so it makes sense to record for the sake of posterity. That said, at the moment there is no known construction that produces strictly $n$-strong heir-coheirs for $n$ in the interval $(1,\omega)$.} so the story for weaves plays out in essentially the same way as the story for NCTP and NATP did in \cite{NCTP}: We are able to get an exact characterization of the presence of $(k,1,1)$-weaves of depth $\omega$ in terms of the failure of a certain form of Kim's lemma and by using (semi-)reliably invariant types we are able to find a closely related form of Kim's lemma that is non-vacuous over arbitrary invariance bases (\cref{thm:main-theorem}), but our results have the same three shortcomings. Firstly, the proof is entirely uniform in $k$, so we are unable to show that these conditions for various $k$ are equivalent. Secondly, the technique does not seem to generalize at all to building $n$-strong heir-coheirs for $n > 1$. Thirdly, there is still no sign of a technique for building a failure of Kim's lemma with regards to a (semi-)reliably invariant type and the precise relationship between (semi-)reliably invariant types and heir-coheirs remains unclear (see Figure~\ref{fig:invariant}).

The motivation for the definition of $k$-grids is merely that they are the simplest configuration the author has found for which he was able to prove \cref{thm:grid-theorem}, providing a combinatorial upper bound (modulo set-theoretic assumptions) on generic stationary local character (introduced in \cite{NCTP}) in addition to the above two mentioned failures of variants of Kim's lemma involving strong heir-coheirs.

The general picture is partially summarized in Figure~\ref{fig:implications} at the end of the paper, although a few properties are missing from the diagram for reasons of space or geometry.

\section{Weaves}%
\label{sec:configuration}

In this section we will define the main combinatorial configuration of this paper and prove some basic properties of it.

\begin{defn}
  A linear order $(L,<)$ is \emph{ordinal-like} if it is a model of the common first-order theory of ordinals%
\end{defn}

Note that any non-maximal element of an ordinal-like linear order has a successor. %

\begin{defn}$ $\label{defn:L-sequence}
  \begin{itemize}
  \item Given an ordinal-like linear order $(L,<)$ and a set $X$, we write $X^L$ for the collection of functions from $L$ to $X$.
  \item A set $A \subset L$ is a \emph{topped initial segment (of $L$)} if it is an initial segment of $L$ and $L \setminus A$ has a least element. Given a topped initial segment $A \subset L$, we write $\top(A)$ for the minimum element of $L \setminus A$.
  \item We write $X^{<L}$ for the collection of partial functions from topped initial segments of $L$ to $X$. We write $X^{\leq L}$ for $X^{< L} \cup X^L$. %
  \item Given $\sigma \in X^{<L}$ and $a \in X$, we write $\sigma \concat a$ for the element of $X^{\leq L}$ satisfying that
    \begin{itemize}
    \item $\dom(\sigma\concat a) = \dom(\sigma) \cup \{\top(\dom(\sigma))\}$,
    \item for any $i \in \dom(\sigma)$, $(\sigma \concat a)(i) = \sigma(i)$, and
    \item $(\sigma \concat a)(\top(\dom(\sigma))) = a$.
    \end{itemize}
  \end{itemize}
\end{defn}

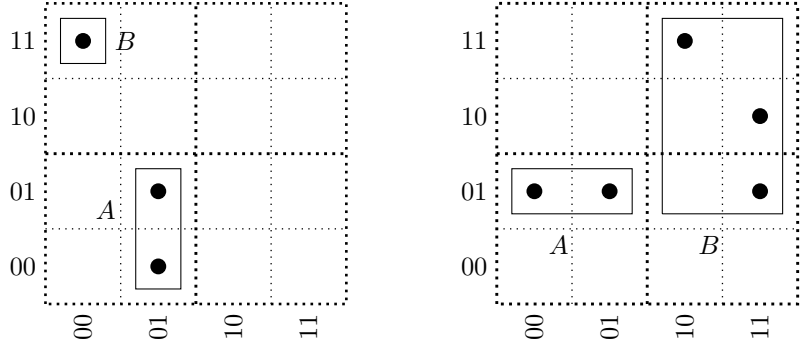
\begin{figure}
  \centering
  \begin{tikzpicture}
    \draw[dotted,line width=1pt] (0,0) -- (0,2) -- (0,4) -- (4,4) -- (4,2) -- (4,0) -- (2,0) -- cycle; %
    \foreach \i in {0,1} {
      \foreach \j in {0,1} {
        \begin{scope}[shift={(2*\i,2*\j)}]
          \draw[dotted,line width=0.5pt] (1,0) -- (1,2);
          \draw[dotted,line width=0.5pt] (0,1) -- (2,1);
        \end{scope}
      }
    }
    \foreach \i in {0,1} {
      \begin{scope}[shift={(2*\i,0)}]
        \foreach \ii in {0,1} {
          \begin{scope}[shift={(\ii,0)}]
            \node[rotate=90,left] at (0.5,0) {\i\ii};
          \end{scope}
        }
      \end{scope}
    }
    \foreach \i in {0,1} {
      \begin{scope}[shift={(0,2*\i)}]
        \foreach \ii in {0,1} {
          \begin{scope}[shift={(0,\ii)}]
            \node[left] at (0,0.5) {\i\ii};
          \end{scope}
        }
      \end{scope}
    }
    \draw[dotted,line width=1.1pt] (0,2) -- (4,2);
    \draw[dotted,line width=1.1pt] (2,0) -- (2,4);
    \fill (1.5,1.5) circle (3pt);
    \fill (1.5,0.5) circle (3pt);
    \draw (1.2,0.2) -- node[left,shift={(-4pt,7pt)}]{$A$} (1.2,1.8) -- (1.8,1.8) -- (1.8,0.2) -- cycle;
    \fill (0.5,3.5) circle (3pt);
    \draw (0.2,3.2) -- (0.2,3.8) -- (0.8,3.8) -- node[right]{$B$} (0.8,3.2) -- cycle;
    \begin{scope}[shift={(6,0)}]
    \draw[dotted,line width=1.1pt] (0,0) -- (0,2) -- (0,4) -- (4,4) -- (4,2) -- (4,0) -- (2,0) -- cycle; %
    \draw[dotted,line width=1.1pt] (0,2) -- (4,2);
    \draw[dotted,line width=1.1pt] (2,0) -- (2,4);
    \foreach \i in {0,1} {
      \foreach \j in {0,1} {
        \begin{scope}[shift={(2*\i,2*\j)}]
          \draw[dotted,line width=0.5pt] (1,0) -- (1,2);
          \draw[dotted,line width=0.5pt] (0,1) -- (2,1);
        \end{scope}
      }
    }
    \fill (0.5,1.5) circle (3pt);
    \fill (1.5,1.5) circle (3pt);
    \draw (0.2,1.2) -- node[below,shift={(-5pt,-5pt)}]{$A$} (1.8,1.2) -- (1.8,1.8) -- (0.2,1.8) -- cycle;
    \fill (3.5,1.5) circle (3pt);
    \fill (3.5,2.5) circle (3pt);
    \fill (2.5,3.5) circle (3pt);
    \draw (2.2,1.2) -- (2.2,3.8) -- (3.8,3.8) -- (3.8,1.2) -- node[below,shift={(-5pt,-5pt)}]{$B$} (2.2,1.2);

    \foreach \i in {0,1} {
      \begin{scope}[shift={(2*\i,0)}]
        \foreach \ii in {0,1} {
          \begin{scope}[shift={(\ii,0)}]
            \node[rotate=90,left] at (0.5,0) {\i\ii};
          \end{scope}
        }
      \end{scope}
    }
    \foreach \i in {0,1} {
      \begin{scope}[shift={(0,2*\i)}]
        \foreach \ii in {0,1} {
          \begin{scope}[shift={(0,\ii)}]
            \node[left] at (0,0.5) {\i\ii};
          \end{scope}
        }
      \end{scope}
    }

    \end{scope}
  \end{tikzpicture}
  \caption{$A$ narrowly below $B$ (left) and $A$ widely to the left of $B$ (right).}%
  \label{fig:narrowly-widely}
\end{figure}

The only set $X$ we will be applying \cref{defn:L-sequence} to is $2^2$, the set of pairs $(i,j)$ with $i,j < 2$. In figures and in the visually motivated terminology in this paper, we will picture $(2^2)^L$ as $2^L \times 2^L$, with the first coordinate horizontal and the second coordinate vertical.

\begin{defn}\label{defn:narrow-wide-comb}
  Given two sets $A,B \subseteq (2^2)^{L}$, we say that
  \begin{itemize}
  \item $A$ is \emph{narrowly below} $B$ (or $B$ is \emph{narrowly above} $A$) if there is a $\tau \in (2^2)^{<L}$ and an $i < 2$ such that every element of $A$ extends $\tau\concat(i,0)$ and every element of $B$ extends $\tau\concat(i,1)$.
  \item $A$ is \emph{narrowly to the left of} $B$ (or $B$ is \emph{narrowly to the right of} $A$) if there is a $\tau \in (2^2)^{<L}$ and a $j < 2$ such that every element of $A$ extends $\tau\concat(0,j)$ and every element of $B$ extends $\tau\concat(1,j)$.
  \item $A$ is \emph{widely to the left of} $B$ (or $B$ is \emph{widely to the right of} $A$) if there is a $\sigma \in (2^2)^{<L}$ such that every element of $A$ extends $\sigma\concat (0,0)$ or $\sigma \concat (0,1)$ and every element of $B$ extends $\sigma\concat (1,0)$ or $\sigma \concat (1,1)$. 
  \end{itemize}
  Define the following classes of finite subsets of $(2^2)^L$ inductively:
  \begin{itemize}
  \item The class of \emph{finite up-$n$-combs} is the smallest class containing the singletons and satisfying that if $A$ and $B$ are a finite up-$n$-combs, $|A| \leq n$, and $A$ is narrowly below $B$, then $A\cup B$ is a finite up-$n$-comb.
  \item The class of \emph{finite right-$n$-combs} is the smallest class containing the singletons and satisfying that if $A$ and $B$ are finite right-$n$-combs, $|A| \leq n$, and $A$ is narrowly to the left of $B$, then $A\cup B$ is a finite right-$n$-comb.
  \item The class of \emph{finite wide right-$n$-combs} is the smallest class containing the singletons and satisfying that if $A$ and $B$ are finite right-$n$-combs, $|A| \leq n$, and $A$ is widely to the left of $B$, then $A\cup B$ is a finite wide right-$n$-comb.
  \end{itemize}
  A \emph{(wide) right-$n$-comb} is a set $A$ satisfying that every finite $A_0 \subseteq A$ is a finite (wide) right-$n$-comb. \emph{Up-$n$-combs} are defined similarly.
\end{defn}

Obviously it would make sense to define the notion of $A$ being `widely below' $B$ as well as the notion of a finite `wide up-$n$-comb,' but we will not use these.

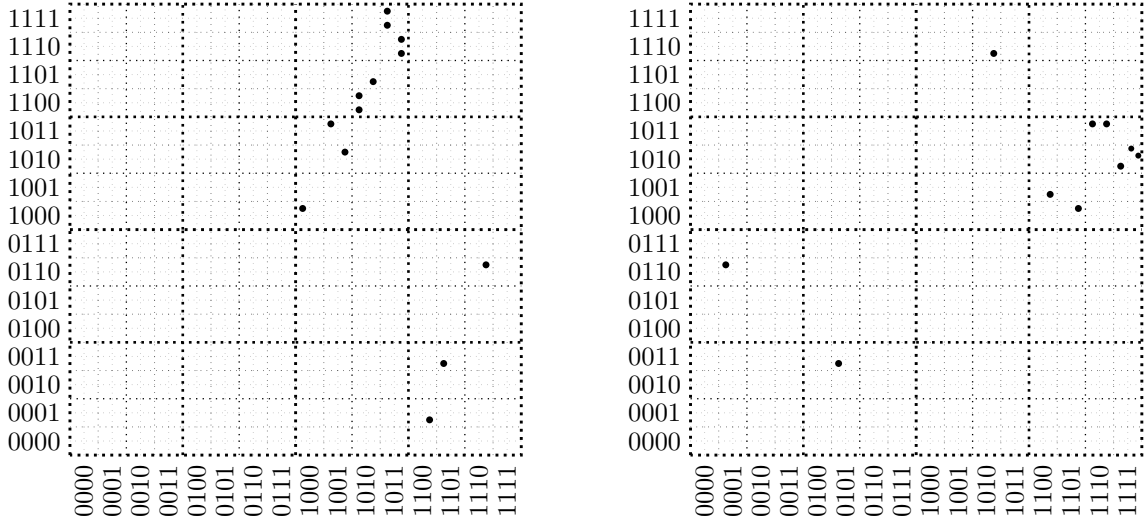
\begin{figure}
  \centering
  \begin{tikzpicture}
    \begin{scope}[scale=0.75,shift={(-11,0)}]
      \draw[dotted,line width=1.1pt] (0,0) -- (0,8) -- (8,8) -- (8,0) -- cycle;
      \draw[dotted,line width=1.1pt] (0,4) -- (8,4);
      \draw[dotted,line width=1.1pt] (4,0) -- (4,8);
      \foreach \i in {0,1} {
        \foreach \j in {0,1} {
          \begin{scope}[shift={(4*\i,4*\j)}]
            \draw[dotted,line width=1pt] (0,2) -- (4,2);
            \draw[dotted,line width=1pt] (2,0) -- (2,4);
            \foreach \ii in {0,1} {
              \foreach \jj in {0,1} {
                \begin{scope}[shift={(2*\ii,2*\jj)}]
                  \draw[dotted,line width=0.5pt] (0,1) -- (2,1);
                  \draw[dotted,line width=0.5pt] (1,0) -- (1,2);
                  \foreach \iii in {0,1} {
                    \foreach \jjj in {0,1} {
                      \begin{scope}[shift={(\iii,\jjj)}]
                        \draw[dotted,line width=0.25pt,opacity=0.75] (0,0.5) -- (1,0.5);
                        \draw[dotted,line width=0.25pt,opacity=0.75] (0.5,0) -- (0.5,1);
                        \foreach \iiii in {0,1} {
                          \foreach \jjjj in {0,1} {
                            \begin{scope}[shift={(0.5*\iiii,0.5*\jjjj)}]
                              \draw[dotted,line width=0.1pt,opacity=0.75] (0,0.25) -- (0.5,0.25);
                              \draw[dotted,line width=0.1pt,opacity=0.75] (0.25,0) -- (0.25,0.5);
                            \end{scope}
                          }
                        }
                      \end{scope}
                    }
                  }
                \end{scope}
              }
            }
          \end{scope}
        }
      }
      \foreach \i in {0,1} {
        \begin{scope}[shift={(4*\i,0)}]
          \foreach \ii in {0,1} {
            \begin{scope}[shift={(2*\ii,0)}]
              \foreach \iii in {0,1} {
                \begin{scope}[shift={(\iii,0)}]
                  \node[left,rotate=90] at (0.25,0) {$\i\ii\iii 0$};
                  \node[left,rotate=90] at (0.75,0) {$\i\ii\iii 1$};
                \end{scope}
              }
            \end{scope}
          }
        \end{scope}
      }
      \foreach \i in {0,1} {
        \begin{scope}[shift={(0,4*\i)}]
          \foreach \ii in {0,1} {
            \begin{scope}[shift={(0,2*\ii)}]
              \foreach \iii in {0,1} {
                \begin{scope}[shift={(0,\iii)}]
                  \node[left] at (0,0.25) {$\i\ii\iii 0$};
                  \node[left] at (0,0.75) {$\i\ii\iii 1$};
                \end{scope}
              }
            \end{scope}
          }
        \end{scope}
      }
      \begin{scope}[xscale=-1,rotate=90]
        \fill ($ (0.75,3.25) + (-0.125,0.125) + (0,3)$) circle (1.75pt);
        \fill ($ (0.75,3.25) + (-0.125,0.125) + (1,3.25)$) circle (1.75pt);
        \fill ($(2.75,1.75) + (-0.125,-0.125) + (0.75,5.75)$) circle (1.75pt);
        \fill ($(5.25,7.25) + (0.125,-0.125) + (-1,-3)$) circle (1.75pt);
        \fill ($(6.25,4.75) + (0.125,-0.125) + (-1,0.25)$) circle (1.75pt);
        \fill ($(6.25,4.75) + (0.125,-0.125) + (-0.5,0)$) circle (1.75pt);
        \fill ($(6.75,4.25) + (0.125,0.125) + (-0.75,0.75)$) circle (1.75pt);
        \fill ($(6.75,4.25) + (0.125,0.125) + (-0.5,0.75)$) circle (1.75pt);
        \fill ($(6.75,4.25) + (0.125,0.125) + (-0.25,1)$) circle (1.75pt);
        \fill (7.125,5.875) circle (1.75pt);
        \fill (7.375,5.875) circle (1.75pt);

        \fill ($(7.5625,5.1875) + (0.125/2,-0.125/2) + (0,0.5)$) circle (1.75pt);
        \fill ($(7.9375,5.3125) + (-0.125/2,0.125/2) + (0,0.25)$) circle (1.75pt);
      \end{scope}
    \end{scope}
    
    \begin{scope}[scale=0.75]
      \draw[dotted,line width=1.1pt] (0,0) -- (0,8) -- (8,8) -- (8,0) -- cycle;
      \draw[dotted,line width=1.1pt] (0,4) -- (8,4);
      \draw[dotted,line width=1.1pt] (4,0) -- (4,8);
      \foreach \i in {0,1} {
        \foreach \j in {0,1} {
          \begin{scope}[shift={(4*\i,4*\j)}]
            \draw[dotted,line width=1pt] (0,2) -- (4,2);
            \draw[dotted,line width=1pt] (2,0) -- (2,4);
            \foreach \ii in {0,1} {
              \foreach \jj in {0,1} {
                \begin{scope}[shift={(2*\ii,2*\jj)}]
                  \draw[dotted,line width=0.5pt] (0,1) -- (2,1);
                  \draw[dotted,line width=0.5pt] (1,0) -- (1,2);
                  \foreach \iii in {0,1} {
                    \foreach \jjj in {0,1} {
                      \begin{scope}[shift={(\iii,\jjj)}]
                        \draw[dotted,line width=0.25pt,opacity=0.75] (0,0.5) -- (1,0.5);
                        \draw[dotted,line width=0.25pt,opacity=0.75] (0.5,0) -- (0.5,1);
                        \foreach \iiii in {0,1} {
                          \foreach \jjjj in {0,1} {
                            \begin{scope}[shift={(0.5*\iiii,0.5*\jjjj)}]
                              \draw[dotted,line width=0.1pt,opacity=0.75] (0,0.25) -- (0.5,0.25);
                              \draw[dotted,line width=0.1pt,opacity=0.75] (0.25,0) -- (0.25,0.5);
                            \end{scope}
                          }
                        }
                      \end{scope}
                    }
                  }
                \end{scope}
              }
            }
          \end{scope}
        }
      }
      \foreach \i in {0,1} {
        \begin{scope}[shift={(4*\i,0)}]
          \foreach \ii in {0,1} {
            \begin{scope}[shift={(2*\ii,0)}]
              \foreach \iii in {0,1} {
                \begin{scope}[shift={(\iii,0)}]
                  \node[left,rotate=90] at (0.25,0) {$\i\ii\iii 0$};
                  \node[left,rotate=90] at (0.75,0) {$\i\ii\iii 1$};
                \end{scope}
              }
            \end{scope}
          }
        \end{scope}
      }
      \foreach \i in {0,1} {
        \begin{scope}[shift={(0,4*\i)}]
          \foreach \ii in {0,1} {
            \begin{scope}[shift={(0,2*\ii)}]
              \foreach \iii in {0,1} {
                \begin{scope}[shift={(0,\iii)}]
                  \node[left] at (0,0.25) {$\i\ii\iii 0$};
                  \node[left] at (0,0.75) {$\i\ii\iii 1$};
                \end{scope}
              }
            \end{scope}
          }
        \end{scope}
      }
      \fill ($ (0.75,3.25) + (-0.125,0.125) $) circle (1.75pt);
      \fill ($(2.75,1.75) + (-0.125,-0.125)$) circle (1.75pt);
      \fill ($(5.25,7.25) + (0.125,-0.125)$) circle (1.75pt);
      \fill ($(6.25,4.75) + (0.125,-0.125)$) circle (1.75pt);
      \fill ($(6.75,4.25) + (0.125,0.125)$) circle (1.75pt);
      \fill (7.125,5.875) circle (1.75pt);
      \fill (7.375,5.875) circle (1.75pt);

      \fill ($(7.5625,5.1875) + (0.125/2,-0.125/2)$) circle (1.75pt);
      \fill (7.8125,5.4375) circle (1.5pt);
      \fill (7.9375,5.3125) circle (1.5pt);
    \end{scope}
  \end{tikzpicture}
  \caption{An up-$3$-comb (left) and a wide right-$2$-comb (right).}
  \label{fig:right-1-comb}
\end{figure}

\begin{defn}\label{defn:weave}
  For $k< \omega$, $m,n \leq \omega$, ordinal-like $L$, and $X \subseteq (2^2)^L$, a \emph{partial $(k,m,n)$-weave for $\varphi(x,y)$ of depth $L$ (on $X$)} is a family $(b_{\sigma} : \sigma \in X)$ of parameters in the sort of $y$ such that
  \begin{itemize}
  \item for any finite up-$m$-comb $C \subseteq X$, $\{\varphi(x,b_\sigma) : \sigma \in C\}$ is $k$-inconsistent and
  \item for any finite right-$n$-comb $C \subseteq X$, $\{\varphi(x,b_\sigma) : \sigma \in C\}$ is consistent.
  \end{itemize}
  \emph{A partial strong $(k,m,n)$-weaves for $\varphi(x,y)$ of depth $L$ (on $X$)} is a partial $(k,m,n)$-weave for $\varphi(x,y)$ of depth $L$ on $X$ satisfying the additional condition (strengthening the second bullet point above) that for any finite wide right-$n$-comb $C \subseteq X$, $\{\varphi(x,b_\sigma) : \sigma \in C\}$ is consistent.

  A partial (strong) $(k,m,n)$-weave $(b_\sigma : \sigma \in X)$ for $\varphi(x,y)$ of depth $L$ is a \emph{(strong) $(k,m,n)$-weave for $\varphi(x,y)$ of depth $L$} if $X = (2^2)^L$.

  A \emph{(partial, strong) $(k,m,n)$-weave of depth $L$} is a (partial, strong) $(k,m,n)$-weave for $\varphi(x,y)$ of depth $L$ for some formula $\varphi(x,y)$.
\end{defn}

It might make sense to refer to up-$\omega$-combs as `vertical antichains' and right-$\omega$-combs as `horizontal antichains,' as the up/down and left/right orientation ceases to be meaningful in that case, but for the sake of attempting to minimize terminology we will not do this.

One thing to note is that there is an important asymmetry between $m$ and $n$ in \cref{defn:weave}.
Specifically, if $(b_\sigma : \sigma \in X)$ is a partial $(k,m,n)$-weave and $m+1 \geq k$, then it is also a partial $(k,\omega,n)$-weave, since any up-$m$-comb of size at most $m+1$ is also an up-$m'$-comb for any $m' \geq m$.

The particular combinatorial consistency-inconsistency configuration we will be considering is that of (strong) $(k,m,n)$-weaves of depth $\omega$. There are two reasons we have bothered with the extra complexity of defining both weaves and strong weaves. The first is that it makes the connection between $(2,1,\omega)$-weaves and cographs discussed in \cref{sec:cographs} cleaner. The second  is that strong weaves are what naturally arise in the proof of \cref{thm:failure-of-Kim-bi-invariant-to-weaves} but (non-strong) weaves are all that we need in the proof of \cref{prop:get-heir-coheir-heir-coheir-from-weave} (see Figure~\ref{fig:U-and-R}). This gives indirectly that a theory $T$ has a $(k,1,1)$-weave of depth $\omega$ if and only if it has a strong $(k,1,1)$-weave of depth $\omega$, but it is worth establishing that this holds for $(k,m,n)$-weaves in general in an attempt to keep the zoo of combinatorial consistency-inconsistency configurations as small as possible.

We will write $f``[X]$ for the image of the set $X$ under the function $f$. %

\begin{figure}
  \centering
  \begin{tikzpicture}
    \begin{scope}[scale=0.75,shift={(-10,0)}]
      \draw[line width=3pt] (0,0) -- (0,8) -- (8,8) -- (8,0) -- cycle;
      \draw[line width=3pt] (0,4) -- (8,4);
      \draw[line width=3pt] (4,0) -- (4,8);
      \foreach \i in {0,1} {
        \foreach \j in {0,1} {
          \begin{scope}[shift={(4*\i,4*\j)}]
            \draw[line width=1.4pt] (0,2) -- (4,2);
            \draw[line width=1.4pt] (2,0) -- (2,4);
            \foreach \ii in {0,1} {
              \foreach \jj in {0,1} {
                \begin{scope}[shift={(2*\ii,2*\jj)}]
                  \draw[line width=0.5pt] (0,1) -- (2,1);
                  \draw[line width=0.5pt] (1,0) -- (1,2);
                \end{scope}
              }
            }
          \end{scope}
        }
      }
      \coordinate (A0000) at (1.1,-0.1);
      \coordinate (A0100) at (3.1,-0.1);
      \coordinate (A1000) at (5.1,-0.1);
      \coordinate (A1100) at (7.1,-0.1);
      \coordinate (B0011) at (1.1,8.1);
      \coordinate (B0111) at (3.1,8.1);
      \coordinate (B1011) at (5.1,8.1);
      \coordinate (B1111) at (7.1,8.1);
    \end{scope}

    \begin{scope}[scale=0.75]
      \draw[dotted,line width=1pt,opacity=0.6] (0,0) -- (0,8) -- (8,8) -- (8,0) -- cycle;
      \draw[dotted,line width=1pt,opacity=0.6] (0,4) -- (8,4);
      \draw[dotted,line width=1pt,opacity=0.6] (4,0) -- (4,8);
      \foreach \i in {0,1} {
        \foreach \j in {0,1} {
          \begin{scope}[shift={(4*\i,4*\j)}]
            \draw[dotted,line width=0.7pt,opacity=0.75] (0,2) -- (4,2);
            \draw[dotted,line width=0.7pt,opacity=0.75] (2,0) -- (2,4);
            \foreach \ii in {0,1} {
              \foreach \jj in {0,1} {
                \begin{scope}[shift={(2*\ii,2*\jj)}]
                  \draw[dotted,line width=0.5pt,opacity=0.75] (0,1) -- (2,1);
                  \draw[dotted,line width=0.5pt,opacity=0.75] (1,0) -- (1,2);
                  \foreach \iii in {0,1} {
                    \foreach \jjj in {0,1} {
                      \begin{scope}[shift={(\iii,\jjj)}]
                        \draw[dotted,line width=0.25pt,opacity=0.75] (0,0.5) -- (1,0.5);
                        \draw[dotted,line width=0.25pt,opacity=0.75] (0.5,0) -- (0.5,1);
                      \end{scope}
                    }
                  }
                \end{scope}
              }
            }
          \end{scope}
        }
      }
      \foreach \i in {0,1} {
        \foreach \j in {0,1} {
          \begin{scope}[shift={(6*\i,2*\j)}]
            \foreach \ii in {0,1} {
              \foreach \jj in {0,1} {
                \begin{scope}[shift={(1.5*\ii,0.5*\jj)}]
                  \coordinate (C\i\ii\j\jj) at (1/4-0.1,-0.1);
                  \coordinate (D\i\ii\j\jj) at (1/4-0.1,1/4+0.1);
                  \foreach \iii in {0,1} {
                    \foreach \jjj in {0,1} {
                      \begin{scope}[shift={(1.5/4*\iii,0.5/4*\jjj)}]
                        \draw[line width=0.7pt] (0,0) -- (0.125,0) -- (0.125,0.125) -- (0,0.125) -- cycle;
                      \end{scope}
                    }
                  }
                \end{scope}
              }
            }
          \end{scope}
        }
      }
    \end{scope}
    \draw[->] (B0011) to [out=25,in=125] (D0011);
    \draw[->] (B0111) to [out=25,in=125] (D0111);
    \draw[->] (B1011) to [out=25,in=125] (D1011);
    \draw[->] (B1111) to [out=25,in=125] (D1111);
    \draw[->] (A0000) to [out=-25,in=-135] (C0000);
    \draw[->] (A0100) to [out=-25,in=-135] (C0100);
    \draw[->] (A1000) to [out=-25,in=-135] (C1000);
    \draw[->] (A1100) to [out=-25,in=-135] (C1100);
  \end{tikzpicture}
  \vspace{-2.5em}
  \caption{The function $f$ in the proof of \cref{prop:weave-iff-strong-weave}. The image of any wide right-$n$-comb under $f$ is a right-$n$-comb, and the image of any up-$m$-comb is an up-$m$-comb.}
  \label{fig:blah}
\end{figure}
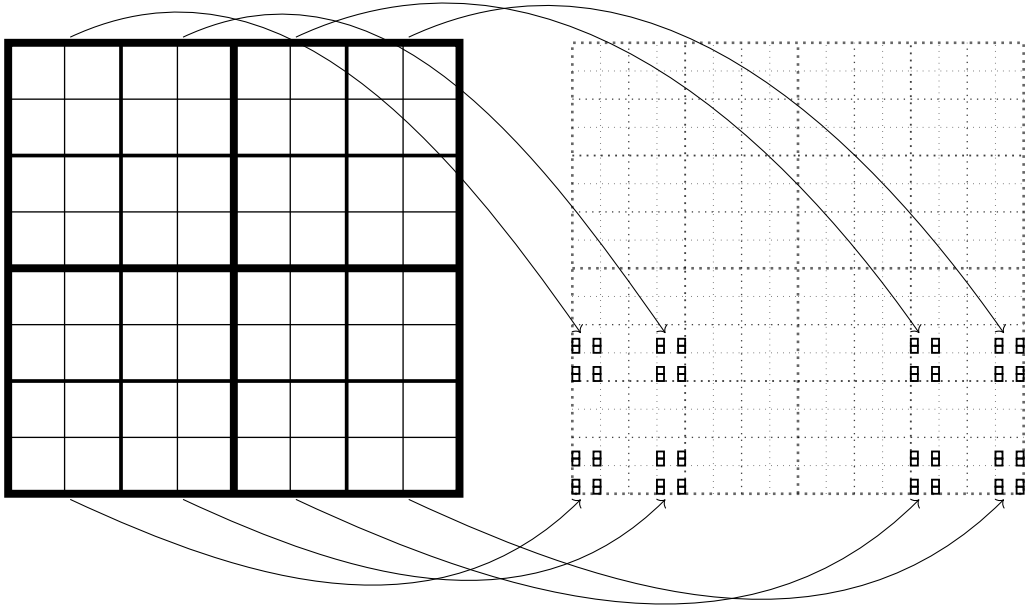

\begin{prop}\label{prop:weave-iff-strong-weave}
  A theory $T$ has a $(k,m,n)$-weave for $\varphi(x,y)$ of depth $\omega$ if and only if it has a strong $(k,m,n)$-weave for $\varphi(x,y)$ of depth $\omega$.
\end{prop}
\begin{proof}
  Assume that $T$ has a $(k,m,n)$-weave $(b_\sigma : \sigma \in (2^2)^\omega)$ for $\varphi(x,y)$. Let $f : (2^2)^\omega \to (2^2)^\omega$ be defined by $f(\alpha)(2n) = (\fst (\alpha(n)),0)$ (where $\fst((i,j)) = i$) and $f(\alpha)(2n+1) = \alpha(n)$. Note that for any $A,B \subseteq (2^2)^\omega$, 
  \begin{itemize}
  \item if $A$ is narrowly below $B$, then $f``[A]$ is narrowly below $f``[B]$ and
  \item if $A$ is narrowly to the left of $B$, then $f``[A]$ is widely to the left of $f``[B]$. 
  \end{itemize}
  It follows from this that the image of any up-$m$-comb under $f$ is an up-$m$-comb and that the image of any right-$n$-comb under $f$ is a wide right-$n$-comb, whereby $(b_{f(\sigma)} : \sigma \in (2^2)^\omega)$ is a strong $(k,m,n)$-weave for $\varphi(x,y)$ of depth $\omega$.

  The other direction follows immediately from the fact that any strong $(k,m,n)$-weave is also a $(k,m,n)$-weave.
\end{proof}

Given \cref{prop:weave-iff-strong-weave}, we will primarily phrase things in terms of weaves, rather than strong weaves (including \cref{thm:failure-of-Kim-bi-invariant-to-weaves}). %

The following definition is a bit more complicated than it needs to be for the purposes of this section (which is showing the relatively routine fact that a theory $T$ has a strong $(k,m,n)$-weave for a given formula $\varphi(x,y)$ of depth $\omega$ if it has strong $(k,m,n)$-weaves for $\varphi(x,y)$ of depth $d$ for every finite $d$), but we will be using this extra machinery later in \cref{sec:converse}.

\begin{defn}
  Given a model $M$ and a $(k,m,n)$-weave $(b_\sigma : \sigma \in X)$ for $\varphi(x,y)$ of depth $L$, the \emph{weave structure associated to $M$ and $(b_\sigma : \sigma \in X)$} is the three-sorted structure $(M,(2^2)^{\leq L},L,B,\prec,<,\eval,|\cdot |)$ where $B : (2^2)^L \to M_n$ is a function\footnote{Strictly speaking this is a partial function whose domain is the definable subset of $(2^2)^{\leq L}$ of elements of maximal height.} satisfying that $B(\sigma) = b^n_\sigma$ for all $\sigma \in (2^2)^L$, $\prec$ is the extension relation on $(2^2)^{\leq L}$, $<$ is the order on $L$, $\eval : (2^2)^{\leq L} \times L \to 2^2$ is the partial evaluation function $\eval(\sigma,i) = \sigma(i)$, and $|\cdot | : (2^2)^{<L} \to L$ is the (partial) height function $|\sigma| = \top(\dom(\sigma))$.

  A \emph{weave structure} is a weave structure associated to some model $M$ and some family $(b_\sigma : \sigma \in X)$.

\end{defn}

The function $\eval$ could be regarded as a literal function to a fourth sort $2^2$ (with constants naming the four elements of $2^2$) or as a pair of predicates giving the value of the first and second coordinates of $\eval(\sigma,i)$.

\begin{defn}\label{defn:weave-model}
  A \emph{$(k,m,n)$-weave model for $\varphi(x,y)$} is three-sorted structure $(M,W,L,B,\prec,<,\eval,|\cdot|)$ that is a model of the common first-order theory of weave structures associated to models of $T$ with $(k,m,n)$-weaves for $\varphi(x,y)$.

  An \emph{unbounded $(k,m,n)$-weave model for $\varphi(x,y)$} is a $(k,m,n)$-weave model $(M,W,L,B,\prec,<,\eval,|\cdot|)$ for $\varphi(x,y)$ such that $L$ has no maximal element.

  Given a weave model $(M,W,L,B,\prec,<,\eval,|\cdot|)$, we'll write $W_{<L}$ for the elements of $W$ in the domain of $|\cdot|$ and $W_{\top}$ for the elements of $W$ not in the domain of $|\cdot|$ (i.e., those $\sigma$ for which $\eval(\sigma,i)$ is defined for all $i \in L$).

\end{defn}

\begin{lem}\label{lem:weave-models-wrok}
  Fix $k < \omega$, $m,n \leq \omega$, and $\varphi(x,y)$. If $(M,W,L,B,\prec,<,\eval,|\cdot |)$ is a $(k,m,n)$-weave model for $\varphi(x,y)$, then there is a canonical identification of $W$ with a subset of $(2^2)^L$ such that $(B(\sigma) : \sigma \in W)$ is a partial $(k,m,n)$-weave for $\varphi(x,y)$.
\end{lem}
\begin{proof}
  Define $\iota : W \to (2^2)^L$ by $\iota(a) = (i \mapsto \eval(a,i))$. The first-order theory of weave structures ensures that this is an injection, so we will identify $W$ with its image under $\iota$.

  It is not difficult to show that for any $n$ and $m$, the set of $m$-tuples in $W$ that enumerate right-$n$-combs (resp.~up-$n$-combs) is first-order definable and therefore that the consistency and inconsistency conditions in the definition of partial $(k,m,n)$-weaves for $\varphi(x,y)$ is axiomatizable in first-order logic.
\end{proof}

\begin{lem}\label{lem:foo-truncate}
  Fix an ordinal-like linear order $L$, ordinal-like initial segment\footnote{Note that $L_0$ is not required to be a topped initial segment of $L$.} $L_0 \subseteq L$, and function $f : (2^2)^{L_0} \to (2^2)^L$ satisfying that for each $\sigma \in (2^2)^{L_0}$, $\sigma \subseteq f(\sigma)$ (i.e., $f(\sigma)$ extends $\sigma$ as a function).

  For any $A \subseteq (2^2)^{L_0}$ and $n \leq \omega$, if $A \subseteq (2^2)^{L_0}$ is an up-$n$-comb (resp.\ right-$n$-comb), then $f``[A] \subseteq (2^2)^L$ is an up-$n$-comb (resp.\ right-$n$-comb).
\end{lem}
\begin{proof}
  It is sufficient to check this for finite $A$. The arguments for right-$n$-combs and up-$n$-combs are essentially the same, so we will just give the argument in the right-$n$-comb case. 

  Suppose that $A,B \subseteq (2^2)^{L_0}$ are finite right-$n$-combs with $|A| \leq n$ such that $A$ is narrowly to the left of $B$. Fix $\sigma \in (2^2)^{L_0}$ witnessing this. Suppose that we already know that $f``[A]$ and $f``[B]$ are right-$n$-combs in $(2^2)^{L}$. Then we have that for some $j < 2$ every element of $f``[A]$ extends $\sigma\concat (0,j)$ and every element of $f``[B]$ extends $\sigma \concat (1,j)$. Therefore $f``[A]$ is narrowly to the left of $f``[B]$ and we have that $f``[A\cup B] = f``[A] \cup f``[B]$ is a right-$n$-comb in $(2^2)^L$.
\end{proof}

\begin{lem}\label{lem:weave-truncation}
  For any natural $k < \omega$, naturals $m,n \leq \omega$, formula $\varphi(x,y)$, ordinal-like $L$, subset $X \subseteq (2^2)^L$, ordinal-like initial segment $L_0 \subseteq L$, subset $X_0 \subseteq (2^2)^{L_0}$, function $f : X_0 \to X$ satisfying that for each $\sigma \in X_0$,  $\sigma \subseteq f(\sigma)$ (i.e., $f(\sigma)$ extends $\sigma$ as a function), and partial $(k,m,n)$-weave $(b_\sigma : \sigma \in X)$ for $\varphi(x,y)$ of depth $L$, we have that $(b_{f(\tau)} : \tau \in X_0)$ is a partial $(k,m,n)$-weave for $\varphi(x,y)$ of depth $L$.
\end{lem}
\begin{proof}
  This is immediate from \cref{lem:foo-truncate}.
\end{proof}

\begin{prop}\label{prop:finite-depth-to-depth-omega}
  For any $k < \omega$, $m,n,\leq \omega$, and formula $\varphi(x,y)$, if $T$ has a $(k,m,n)$-weave of depth $d$ for $\varphi(x,y)$ for every $d<\omega$, then $T$ has a $(k,m,n)$-weave for $\varphi(x,y)$ of depth $\omega$.
\end{prop}
\begin{proof}
  For each $d$, fix $M_d \models T$ and a $(k,m,n)$-weave $(b^n_\sigma : \sigma \in (2^2)^d)$ for $\varphi(x,y)$ with each $b^d_\sigma$ an element of $M_d$. For each $d$, let $N_d= (M_d,(2^2)^d,\{0,1,\dots,d-1\},P_{(2^2)^d},B_d,\prec_d,<_d,\eval_d,|\cdot |_d)$ be the weave structure associated to $M_d$ and $(b^d_\sigma : \sigma \in (2^2)^n)$. Let $N=(M,W,L,B,\prec,<,\eval,|\cdot |)$ be a non-principal ultraproduct of the $M_d$'s. By \cref{lem:weave-models-wrok}, we can identify $W$ with a subset of $(2^2)^{L}$ in a canonical way and moreover if we define $b_\sigma \coloneq B(\sigma)$ for $\sigma \in W$, we have that $(b_\sigma : \sigma \in W)$ is a partial $(k,m,n)$-weave for $\varphi(x,y)$. 

  By construction, $L$ has an initial segment isomorphic to $\omega$, which we will identify with $\omega$. By $\aleph_0$-saturation, we have that for each $\sigma \in (2^2)^{\omega}$, there is a $\tau \in W \subseteq (2^2)^L$ extending $\sigma$. Let $f : (2^2)^\omega \to W$ be a function satisfying that for each $\sigma \in (2^2)^{\omega}$, $\sigma \subseteq f(\sigma)$. By \cref{lem:weave-truncation}, we have that $(b_{f(\sigma)} : \sigma \in (2^2)^\omega)$ is a $(k,m,n)$-weave for $\varphi(x,y)$ of depth $\omega$.
\end{proof}

\section{Weaves from the failure of variants of Kim's lemma}
\label{sec:weaves-from-failures}

Given the large number of variants of Kim's lemma we will be considering, we need to introduce some systematic terminology for them.

\begin{defn}\label{defn:Kim's-lemmas}
  For any classes $\Xc$ and $\Yc$ of pairs $(A,p)$ with $A$ a small set of parameters and $p$ an $A$-invariant type and any $k < \omega$, we say that $T$ satisfies \emph{$(k,\Xc,\Yc)$-\KL{}} if for any set of parameters $A$ in a model of $T$, formula $\varphi(x,b)$, and $p(y),q(y) \supset \tp(b/M)$ with $(A,p) \in \Xc$ and $(A,q) \in \Yc$, if $\varphi(x,b)$ $k$-divides along $p$, then it divides along $q$.

  $T$ satisfies \emph{$(\omega,\Xc,\Yc)$-\KL{}} if it satisfies $(k,\Xc,\Yc)$-\KL{} for all $k < \omega$.

  For $k\leq \omega$ and class $\Zc$ of small sets of parameters (e.g., models, invariance bases), $T$ satisfies \emph{$(k,\Xc,\Yc)$-\KL{} over $\Zc$} if it satisfies $(k,\{ (A,p) \in \Xc:A \in \Zc\},\{(A,p) \in \Yc: A \in \Zc\})$-\KL{}.
\end{defn}

This notion has an easy monotonicity property, which is worth stating explicitly.

\begin{prop}\label{prop:Kim-monotone}
  If $\Xc \subseteq \Xc'$, $\Yc \subseteq \Yc'$, and $k \leq k' \leq \omega$, then $(k',\Xc',\Yc')$-\KL{} implies $(k,\Xc,\Yc)$-\KL{}.
\end{prop}
\begin{proof}
  Suppose that $T$ satisfies $(k',\Xc',\Yc')$-\KL{}. Fix a set of parameters $A$, a formula $\varphi(x,b)$, and invariant types $p(y),q(y) \supset \tp(b/A)$. Suppose that $(A,p) \in \Xc$, $(A,q) \in \Yc$, and $\varphi(x,b)$ $k$-divides along $q$. Then it also $k'$-divides along $q$. Moreover, since $(A,q) \in \Xc$, it is in $\Xc'$ as well, so by assumption we have that $\varphi(x,b)$ divides along every $A$-invariant type $r$ with $(A,r) \in \Yc'$, and so in particular divides along $p$.
\end{proof}

Rather than introduce symbolic notation for various classes of invariant types, we will represent these classes with descriptive phrases. So, for example, the New Kim's Lemma of \cite{NKL} is `$(\omega$, invariant, Kim-strictly invariant$)$-\KL{} over models' in our nomenclature. %

\begin{defn}\label{defn:semi-reliability}
  A sequence $(b_i : i < n)$ is an \emph{invariant sequence over $A$} if $b_i \equiv_A b_j$ for each $i < j < n$ and $b_i \indi_A b_{<i}$ for each $i < n$.

  Given a class of $A$-invariant types $\Ic$, an $A$-invariant type $p(x)$ is \emph{semi-reliably in $\Ic$} if it is in the largest class $\Rc \subseteq \Ic$ satisfying that for any $p(x_0) \in \Rc$ and $q(x_0,\dots,x_{n-1}) \in S(A)$ extending $(p\res A)(x_0)$, if $q(\xbar)$ is the type of an invariant sequence over $A$, then there is an $r(\xbar) \in \Rc$ extending $p(x_0) \cup \dots \cup p(x_{n-1}) \cup q(x_0,\dots,x_{n-1})$.

  If $\Ic$ is the class of all $A$-invariant types and $p(x)$ is reliably in $\Ic$, then we say that $p(x)$ is \emph{semi-reliably $A$-invariant}. If $\Ic$ is the class of $A$-coheirs and $p(x)$ is reliably in $\Ic$, we say that $p(x)$ a \emph{semi-reliable $A$-coheir}.
\end{defn}

\begin{fact}[{\cite[Thm.~2.14]{NCTP}}]\label{fact:reliable-existence}
  Any type over an invariance base $A$ extends to a semi-reliably $A$-invariant type. Any type over a model $M$ extends to a semi-reliable $M$-coheir.
\end{fact}

The concept given here in \cref{defn:semi-reliability} is only called `semi-reliability' because the special invariant types built in \cite[Thm.~2.14]{NCTP} actually satisfy a stronger property (called there `reliability'), but both in the argument there and in (most of) the proofs here, only semi-reliability is actually used. If it turns out that semi-reliability really is the more useful notion, it may make sense to change terminology to keep the names of the most used concepts short (perhaps by calling semi-reliability `reliability' and reliability something else).

Although it is clear that both semi-reliable invariance and bi-invariance imply Kim-strict invariance, it is not clear at the moment what other implications hold. %

\begin{quest}\label{quest:bi-invariant-reliably-invariant}
  Is every bi-invariant type semi-reliably invariant? Is Kim-strict invariance equivalent to semi-reliable invariance?
\end{quest}

An analogous question for heir-coheirs was asked in \cite[Quest.~2.10]{NCTP}.

In the proof of the following proposition (and elsewhere in the paper), when we are dealing with a type $p(\xbar)$ in which the variables $\xbar$ are naturally understood as some family $(x_i : i \in I)$ indexed by some set $I$, we will denote this by $p(x_i : i \in I)$. %

\begin{figure}
  \centering
  \begin{tikzpicture}
    \begin{scope}[scale=1.5]
      \draw[dotted,line width=0.75pt] (0,0) -- (2,0) -- (2,2) -- (0,2) -- cycle;
      \begin{scope}[shift={(1,1)}]
        \draw[line width=1.5pt] (0,0) -- (1,0) -- (1,1) -- node[above] {$(b^d_\sigma : \sigma \in (2^2)^d)$} (0,1) -- cycle;
        \draw[line width=1pt] (0,0.5) -- (1,0.5);
        \draw[line width=1pt] (0.5,0) -- (0.5,1);
        \coordinate (A) at (1.2,0.5);
        \foreach \ii in {0,1} {
          \foreach \jj in {0,1} {
            \begin{scope}[shift={(0.5*\ii,0.5*\jj)}]
              \draw[line width=0.5pt] (0,0.25) -- (0.5,0.25);
              \draw[line width=0.5pt] (0.25,0) -- (0.25,0.5);
              \foreach \iii in {0,1} {
                \foreach \jjj in {0,1} {
                  \begin{scope}[shift={(0.25*\iii,0.25*\jjj)}]
                    \draw[line width=0.5pt,opacity=0.25] (0,0.125) -- (0.25,0.125);
                    \draw[line width=0.5pt,opacity=0.25] (0.125,0) -- (0.125,0.25);
                    \foreach \iiii in {0,1} {
                      \foreach \jjjj in {0,1} {
                        \begin{scope}[shift={(0.125*\iiii,0.125*\jjjj)}]
                          \draw[line width=0.5pt,opacity=0.1] (0,0.0625) -- (0.125,0.0625);
                          \draw[line width=0.5pt,opacity=0.1] (0.0625,0) -- (0.0625,0.125);
                        \end{scope}
                      }
                    }
                  \end{scope}
                }
              }
            \end{scope}
          }
        }
      \end{scope}
      \begin{scope}[shift={(4,0)}]
        \draw[dotted,line width=0.75pt] (0,0) -- (2,0) -- (2,2) -- (0,2) -- cycle;
        \draw[line width=1.5pt] (1,0) -- (1,2) -- (2,2) -- (2,0) -- node[below] {$(b^{d+1}_{(0,i)\concat\sigma} : \sigma \in (2^2)^d,~i<2)$} cycle;
        \draw[line width=1.5pt] (1,1) -- (2,1);
        \coordinate (B) at (0.8,0.5);
        \coordinate (C) at (0.8,1.5);
        \coordinate (D) at (2.2,1);
        \foreach \i in {0,1} {
          \begin{scope}[shift={(1,\i)}]
            \draw[line width=1pt] (0,0.5) -- (1,0.5);
            \draw[line width=1pt] (0.5,0) -- (0.5,1);
            \foreach \ii in {0,1} {
              \foreach \jj in {0,1} {
                \begin{scope}[shift={(0.5*\ii,0.5*\jj)}]
                  \draw[line width=0.5pt] (0,0.25) -- (0.5,0.25);
                  \draw[line width=0.5pt] (0.25,0) -- (0.25,0.5);
                  \foreach \iii in {0,1} {
                    \foreach \jjj in {0,1} {
                      \begin{scope}[shift={(0.25*\iii,0.25*\jjj)}]
                        \draw[line width=0.5pt,opacity=0.25] (0,0.125) -- (0.25,0.125);
                        \draw[line width=0.5pt,opacity=0.25] (0.125,0) -- (0.125,0.25);
                        \foreach \iiii in {0,1} {
                          \foreach \jjjj in {0,1} {
                            \begin{scope}[shift={(0.125*\iiii,0.125*\jjjj)}]
                              \draw[line width=0.5pt,opacity=0.1] (0,0.0625) -- (0.125,0.0625);
                              \draw[line width=0.5pt,opacity=0.1] (0.0625,0) -- (0.0625,0.125);
                            \end{scope}
                          }
                        }
                      \end{scope}
                    }
                  }
                \end{scope}
              }
            }
          \end{scope}
        }
      \end{scope}
      \begin{scope}[shift={(8,0)}]
        \draw[line width=1.5pt] (0,0) -- node[below] {$(b^{d+1}_\sigma : \sigma \in (2^2)^{d+1})$} (2,0) -- (2,2) -- (0,2) -- cycle;
        \draw[line width=1.5pt] (1,0) -- (1,2);
        \draw[line width=1.5pt] (0,1) -- (2,1);
        \coordinate (E) at (-0.2,1);
        \coordinate (F) at (1.5,2.2);
        \foreach \i in {0,1} {
          \foreach \j in {0,1} {
            \begin{scope}[shift={(\i,\j)}]
              \draw[line width=1pt] (0,0.5) -- (1,0.5);
              \draw[line width=1pt] (0.5,0) -- (0.5,1);
              \foreach \ii in {0,1} {
                \foreach \jj in {0,1} {
                  \begin{scope}[shift={(0.5*\ii,0.5*\jj)}]
                    \draw[line width=0.5pt] (0,0.25) -- (0.5,0.25);
                    \draw[line width=0.5pt] (0.25,0) -- (0.25,0.5);
                    \foreach \iii in {0,1} {
                      \foreach \jjj in {0,1} {
                        \begin{scope}[shift={(0.25*\iii,0.25*\jjj)}]
                          \draw[line width=0.5pt,opacity=0.25] (0,0.125) -- (0.25,0.125);
                          \draw[line width=0.5pt,opacity=0.25] (0.125,0) -- (0.125,0.25);
                          \foreach \iiii in {0,1} {
                            \foreach \jjjj in {0,1} {
                              \begin{scope}[shift={(0.125*\iiii,0.125*\jjjj)}]
                                \draw[line width=0.5pt,opacity=0.1] (0,0.0625) -- (0.125,0.0625);
                                \draw[line width=0.5pt,opacity=0.1] (0.0625,0) -- (0.0625,0.125);
                              \end{scope}
                            }
                          }
                        \end{scope}
                      }
                    }
                  \end{scope}
                }
              }
            \end{scope}
          }
        }
      \end{scope}

    \draw[->] (A) -- (C);
    \draw[->] (A) -- node[sloped,below] {clone using $p$} (B);
    \draw[->] (D) -- node[below] {clone using $q$} (E);
    \draw[->] (D) to [out = 45, in = 100] (F);
    \end{scope}
  \end{tikzpicture}
  \caption{The construction in the proof of \cref{thm:failure-of-Kim-bi-invariant-to-weaves}. In the first step, every up-$m$-comb (with $m=1$ for semi-reliably invariant $p$) of size $\ell \leq m$ in the lower clone square realizes $p^{\otimes \ell}$ over the original upper square. In the second step, every right-$n$-comb (with $n=1$ for semi-reliably invariant $q$) of size $\ell \leq n$ in the clone rectangle on the left realizes $q^{\otimes \ell}$ over the original rectangle on the right.}
  \label{fig:construction}
\end{figure}
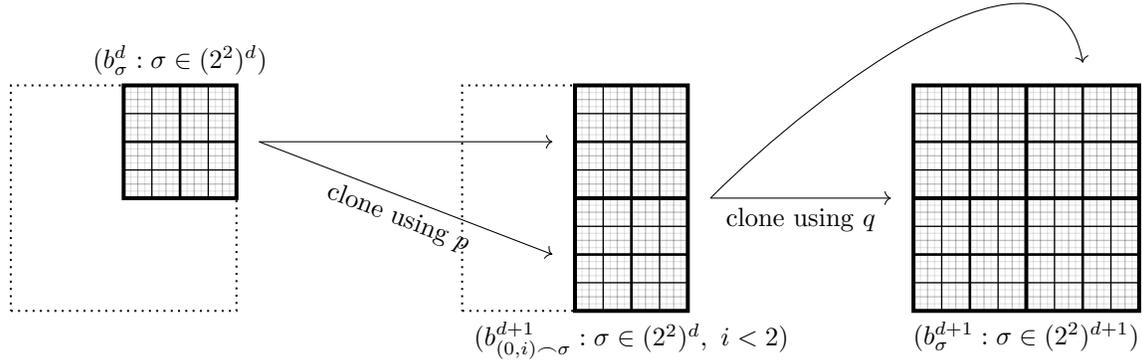

\begin{thm}\label{thm:failure-of-Kim-bi-invariant-to-weaves}
  Fix a theory $T$, $k < \omega$, and $m,n \leq \omega$.
  \begin{enumerate}
  \item\label{bi-bi-fail} If $T$ fails to satisfy \textup{$(k$, $m$-strongly bi-invariant, $n$-strongly bi-invariant$)$}-\KL{}, then $T$ has a $(k,m,n)$-weave of depth $\omega$.
  \item\label{bi-reliable-fail} If $T$ fails to satisfy \textup{$(k$, $m$-strongly bi-invariant, semi-reliably invariant$)$}-\KL{}, then $T$ has a $(k,m,1)$-weave of depth $\omega$.
  \item\label{reliable-bi-fail} If $T$ fails to satisfy \textup{$(k$, semi-reliably invariant, $n$-strongly bi-invariant$)$}-\KL{}, then $T$ has a $(k,1,n)$-weave of depth $\omega$.
  \item\label{reliable-reliable-fail} If $T$ fails to satisfy \textup{$(k$, semi-reliably invariant, semi-reliably invariant$)$}-\KL{}, then $T$ has a $(k,1,1)$-weave of depth $\omega$.
  \end{enumerate}
\end{thm}
\begin{proof}[Proof of (\ref{bi-bi-fail})]
  Fix a formula $\varphi(x,b)$, a set of parameters $A$, and $A$-invariant types $p(y),q(y) \supset \tp(b/A)$ that witness the failure of \textup{$(k,~m$-strongly bi-invariant, $n$-strongly bi-invariant,$)$-\KL{}}. In particular, $p$ is $m$-strongly $A$-bi-invariant, $q$ is $n$-strongly $A$-bi-invariant, $\varphi(x,b)$ $k$-divides along $p$ but does divide along $q$.

  We will prove by induction on $d$ that $T$ has a family $(b^d_\sigma : \sigma \in (2^2)^d)$ of realizations of $\tp(b/A)$ satisfying the following properties:
  \begin{itemize}
  \item[$(\mathtt{U})$] For each up-$m$-comb $C \subseteq (2^2)^d$, $\{b^d_\sigma : \sigma \in C\}$ is a Morley sequence in $p$.
  \item[$(\mathtt{R})$] For each right-$n$-comb $C \subseteq (2^2)^d$, $\{b^d_\sigma : \sigma \in C\}$ is a Morley sequence in $q$.
  \end{itemize}
  This is clearly trivial in the case of $d=0$. Suppose that we have a family $(b^d_\sigma : \sigma \in (2^2)^d)$ satisfying $(\mathtt{R})$ and $(\mathtt{U})$ for some $d < \omega$. Let $b^{d+1}_{(1,1)\concat \sigma} = b^d_\sigma$ for each $\sigma \in (2^2)^d$. Let $\ebar \models p^{\otimes m} \res A \cup ( b^{d+1}_{(1,1)\concat \sigma} : \sigma \in (2^2)^d)$. Since $p$ is $m$-strongly $A$-bi-invariant, we have that $( b^{d+1}_{(1,1)\concat \sigma} : \sigma \in (2^2)^d) \indi_A \ebar$. Let $r(y_\sigma : \sigma \in (2^2)^d)$ be an $A$-invariant type extending $\tp((b^{d+1}_{(1,1)\concat\sigma} : \sigma \in (2^2)^d) / A\ebar)$. Find a family $(b^{d+1}_{(1,0)\concat\sigma} : \sigma \in (2^2)^d) \equiv_A ( b^{d+1}_{(1,1)\concat\sigma} : \sigma \in (2^2)^d)$ satisfying that $( b^{d+1}_{(1,1)\concat \sigma} : \sigma \in (2^2)^d) \models r \res A \cup (b^{d+1}_{(1,0)\concat \sigma} : \sigma \in (2^2)^d)$.

  Now consider the family $(b^{d+1}_{(1,i)\concat\sigma} : \sigma \in (2^2)^d,~i < 2)$. We need to verify that this family satisfies $(\mathtt{U})$. Fix an up-$m$-comb $C \subseteq \{(1,i)\concat \sigma : \sigma\in (2^2)^d,~i<2\}$. Let $C_i = \{ (1,i)\concat\sigma : (1,i)\concat\sigma \in C\}$ for both $i < 2$. It must be the case that $C_0$ and $C_1$ are both up-$m$-combs and moreover that $|C_0| \leq m$. In particular, this implies that (in some enumeration), $(b^{d+1}_{\sigma} : \sigma \in C_0) \models p^{\otimes |C_0|} \res A \cup ( b^{d+1}_{(1,1)\concat\sigma} : \sigma \in (2^2)^d)$ and so a fortiori $(b^{d+1}_{\sigma} : \sigma \in C_0) \models p^{\otimes |C_0|} \res A \cup \{b^{d+1}_{\sigma} : \sigma \in C_1\}$. Therefore $(b^{d+1}_{\sigma} : \sigma \in C_0 \cup C_1)$ is (in some enumeration) a Morley sequence in $p$ by the induction hypothesis.

  By essentially the same argument (with $q$ in place of $p$), we can extend $(b^{d+1}_{(1,i)\concat\sigma} : \sigma \in (2^2)^d,~i < 2)$ to a family $(b^{d+1}_{\sigma} : \sigma \in (2^2)^{d+1})$ such that
  \begin{itemize}
  \item $(b^{d+1}_{ (0,i)\concat\sigma} : \sigma \in (2^2)^d,~i < 2) \equiv_A (b^{d+1}_{(1,i)\concat\sigma} : \sigma \in (2^2)^d,~i < 2)$ and
  \item for any right-$n$-comb $C \subseteq \{(0,i)\concat\sigma : \sigma \in (2^2)^d,~i<2\}$, $(b^{d+1}_{\sigma} : \sigma \in C)$ is (in some enumeration) a Morley sequence in $q$ of length $|C|$ over $A \cup \{b^{d+1}_{(1,i)\concat\sigma} : \sigma \in (2^2)^d,~i<2\}$.
  \end{itemize}
  The first bullet implies that the full family $(b^{d+1}_{\sigma} : \sigma \in (2^2)^{d+1})$ satisfies $(\mathtt{U})$ (since an up-$n$-comb in $(2^2)^{d+1}$ must be contained entirely in either $\{(0,i)\concat\sigma : \sigma \in (2^2)^d,~i<2\}$ or $\{(1,i)\concat\sigma : \sigma \in (2^2)^d,~i<2\}$). The second bullet implies that the full family satisfies $(\mathtt{R})$, so we are done.

  The condition that $\varphi(x,b)$ $k$-divides along $p$ but does not divide along $q$ implies that each of the families $(b^d_\sigma : \sigma \in (2^2)^d)$ is a $(k,m,n)$-weave for $\varphi(x,y)$ of depth $d$. Therefore by \cref{prop:finite-depth-to-depth-omega}, $T$ admits a $(k,m,n)$-weave of depth $\omega$.
\end{proof}

\begin{proof}[Proof of {(\ref{bi-reliable-fail})}]
  Again fix a formula $\varphi(x,b)$, a set of parameters $A$, an $m$-strongly $A$-bi-invariant type $p(y)$, and a reliably $A$-invariant type $q(y)$ such that $\varphi(x,b)$ $k$-divides along $p$ but does not divide along $q$.

  The argument here is very similar to the proof of (\ref{bi-bi-fail}), with some additional bookkeeping needed to manage the semi-reliably invariant type. We will build by induction on $d$ families $(b_\sigma^d : \sigma \in (2^2)^d)$ of realizations of $\tp(b/A)$ satisfying the properties $(\mathtt{U})$ and $(\mathtt{R})$ (with $n = 1$). We will also build a sequence of semi-reliably $A$-invariant types $q_d(y_\sigma : \sigma \in (2^2)^d)$ with the property that the restriction of $q_d$ to each variable $y_\sigma$ is $q(y_\sigma)$.

  For $d=0$, we just take $b_\varnothing^0$ to be $b$ and $q_0(y_\varnothing)$ to be $q(y_\varnothing)$. Suppose we have a family $(b^d_\sigma : \sigma \in (2^2)^d)$ satisfying $(\mathtt{U})$ and $(\mathtt{R})$ as well as a semi-reliably $A$-invariant type $q_d(y_\sigma : \sigma \in (2^2)^d)$ with the property that the restriction of $q_d$ to each variable $y_\sigma$ is $q(y_\sigma)$.

  Let $b^{d+1}_{(1,1)\concat\sigma} = b^d_\sigma$ for each $\sigma \in (2^2)^d$. Find an $A$-invariant type $r(y_\sigma : \sigma \in (2^2)^d)$ in the same manner as in the proof of (\ref{bi-bi-fail}) and similarly build the family $(b^{d+1}_{(1,i)\concat\sigma} : \sigma \in (2^2)^d,~i < 2)$. Since $q_d$ is a semi-reliably $A$-invariant type and since the two-element sequence of tuples $((b^{d+1}_{ (1,1)\concat\sigma} : \sigma \in (2^2)^d),(b^{d+1}_{(1,0)\concat\sigma} : \sigma \in (2^2)^d))$ is an invariant sequence, we can extend
  \[
    q_d(y_{(1,1)\concat\sigma} : \sigma \in (2^2)^d) \cup q_d(y_{(1,0)\concat\sigma} : \sigma \in (2^2)^d) \cup \tp((b^{d+1}_{ (1,1)\concat\sigma} : \sigma \in (2^2)^d),(b^{d+1}_{ (1,0)\concat\sigma} : \sigma \in (2^2)^d)/A)
  \]
  to a semi-reliably $A$-invariant type $q_{d+\nicefrac{1}{2}}(y_{(1,i)\concat\sigma} : \sigma \in (2^2)^d,~i < 2)$ with the property that the restriction of $q_{d+\nicefrac{1}{2}}$ to each $y_{(1,i)\concat\sigma}$ is $q(y_{ (1,i)\concat\sigma})$.

  Now pick $(b^{d+1}_{ (0,i)\concat\sigma} : \sigma \in (2^2)^d,~i < 2) \models q_{d+\nicefrac{1}{2}} \res A \cup (b^{d+1}_{ (1,i)\concat\sigma} : \sigma \in (2^2)^d,~i < 2)$ and collect these into the family $(b^{d+1}_{\sigma} : \sigma \in (2^2)^{d+1})$. By construction and the induction hypothesis we have that for any wide right-$1$-comb $C \subseteq (2^2)^{d+1}$, $\{b^{d+1}_{\sigma} : \sigma \in C\}$ is a Morley sequence in $q$ (in some order), so $(\mathtt{R})$ holds. Moreover, by the same argument as in the proof of (\ref{bi-bi-fail}), we have that $(\mathtt{U})$ holds. Finally, since the two-element sequence of tuples $((b^{d+1}_{ (1,i)\concat\sigma} : \sigma \in (2^2)^d,~i < 2),(b^{d+1}_{ (0,i)\concat\sigma} : \sigma \in (2^2)^d,~i < 2))$ is an invariant sequence, we can extend
  \begin{align*}
    q_{d+\nicefrac{1}{2}}(y_{(1,i)\concat\sigma} : \sigma \in (2^2)^d,~i < 2) &\cup q_{d+\nicefrac{1}{2}}(y_{ (0,i)\concat\sigma} : \sigma \in (2^2)^d,~i < 2) \\ &\cup \tp((b^{d+1}_{(1,i)\concat\sigma} : \sigma \in (2^2)^d,~i < 2),(b^{d+1}_{(0,i)\concat\sigma} : \sigma \in (2^2)^d,~i < 2) / A)
  \end{align*}
to a semi-reliably $A$-invariant type $q_{d+1}(y_{\sigma} : \sigma \in (2^2)^{d+1})$ with the property that the restriction of $q_{d+1}$ to each variable $y_\sigma$ is $q(y_\sigma)$.

The rest of the argument is now the same as in the proof of (\ref{bi-bi-fail}).
\end{proof}

\begin{proof}[Proof of (\ref{reliable-bi-fail}) and (\ref{reliable-reliable-fail})]
  The proofs in these two cases are the same as the proof of (\ref{bi-reliable-fail}), mutatis mutandis.
\end{proof}

As noted earlier, the proof of \cref{thm:failure-of-Kim-bi-invariant-to-weaves} actually gives a strong $(k,m,n)$-weave of depth $\omega$, rather than just a $(k,m,n)$-weave of depth $\omega$, but as we saw in \cref{prop:weave-iff-strong-weave}, these are equivalent anyway.

It seems likely (using ideas from \cite{Mutchnik-NSOP2}) that the statement of \cref{thm:failure-of-Kim-bi-invariant-to-weaves} still holds with $n$-strong heir-coheirs in place of $n$-strongly bi-invariant types and canonical coheirs in place of semi-reliably invariant types---for instance the failure of ($k$, canonical coheir, $n$-strong heir-coheir)--Kim's lemma over models should entail the existence of a $(k,1,n)$-weave of depth $\omega$---but we have not pursued this here.

In the absence of a positive answer to \cref{quest:bi-invariant-reliably-invariant}, it's worth pointing out the following (admittedly awkward) corollary of \cref{thm:failure-of-Kim-bi-invariant-to-weaves}.

\begin{cor}\label{cor:awkward}
  Fix a complete first-order theory $T$ and $k < \omega$.
  \begin{itemize}
  \item If $T$ fails \textup{($k$, bi-invariant or semi-reliably invariant, $n$-strongly bi-invariant)}-\KL{}, then $T$ has a $(k,1,n)$-weave of depth $\omega$.
  \item If $T$ fails \textup{($k$, $m$-strongly bi-invariant, bi-invariant or semi-reliably invariant)}-\KL{}, then $T$ has a $(k,m,1)$-weave of depth $\omega$.
\item If $T$ fails \textup{($k$, bi-invariant or semi-reliably invariant, bi-invariant or semi-reliably invariant)}--Kim's lemma, then $T$ has a $(k,1,1)$-weave of depth $\omega$.
  \end{itemize}

\end{cor}
\begin{proof}
  This is immediate from the definition of $(k,\Xc,\Yc)$-\KL{} and \cref{thm:failure-of-Kim-bi-invariant-to-weaves}.
\end{proof}

\section{The converse for $(k,1,1)$-weaves}
\label{sec:converse}

The argument here is similar to arguments in \cite{NCTP}, but we will only give a proof analogous to that of \cite[Prop.~3.1]{NCTP} (which works for both countable and uncountable languages but is more technical). For countable languages, a proof analogous to that of \cite[Prop.~1.5]{NCTP} (in which the $W$ and $L$ sorts are kept fixed) is also possible. %

Like the proof of \cite[Prop.~3.1]{NCTP}, the argument used here is a `forcing plus compactness' argument. In other words, we have some poset $W$ on which we would like to build a sufficiently generic filter (i.e., one meeting some family of dense requirements). The issue is that we don't know that $W$ is $\kappa$-closed for any $\kappa > \aleph_0$, so to deal with this, we take the poset and our partially built filter $(P_i : i \in I)$, bundle them together in a single first-order structure (with each $P_i$ given its own symbol in the language), and pass to elementary extensions (expanding both the poset $W$ and the individual $P_i$'s in the partially built generic filter) in order to make the intersection of the existing $P_i$'s non-empty (at which point we also expand the language by adding a new $P_i$ symbol for a subset of the intersection). In doing so, new requirements show up (since in our case these correspond to formulas with parameters in the model we are building), but we are able to catch our tail (even when the size of the language is a singular cardinal), as all of the requirements are finitary in nature. Since the requirements are moreover axiomatizable in first-order logic, they remain satisfied even when passing to elementary extensions.

We will take the opportunity to give a general framework for these kinds of arguments. This framework is very similar to something like (a higher-cardinality generalization of) the Rasiowa-Sikorski lemma or the Baire category theorem, but we will use category-theoretic (rather than order-theoretic) language for a little bit of extra flexibility.  %

\begin{defn}
  A category $\Cc$ has \emph{${<}\lambda$-sequential colimits} if for any ordinal $\alpha < \lambda$, any diagram $f : \alpha \to \Cc$ has a colimit.

  $\Cc$ has \emph{$\lambda$-sequential colimits} if it has ${<}\lambda^+$-sequential colimits (i.e., the above holds for any $\alpha \leq \lambda$).
\end{defn}

\begin{defn}\label{defn:generic-above}
  Given a small category $\Cc$ and an object $a \in X$, a set $X$ of morphisms in $\Cc$ with domain $a$ is \emph{generic above $a$} if for every morphism $f : a \to b$, there is an object $c \in \Cc$ and a morphism $g : b \to c$ such that $g \circ f \in X$.
\end{defn}

Note that in the following proposition, $F : \lambda \to \Cc$ being a sequential-colimit-preserving functor just means that for any limit ordinal $\alpha < \lambda$, $F(\alpha)$ is the colimit of the diagram $F \res \alpha$. For something like a category of models with elementary embeddings, this is the same thing as a continuous elementary chain.

\begin{prop}\label{prop:Baire-category-category}
  Fix an infinite cardinal $\lambda$. Let $\Cc$ be a small category with ${<}\lambda$-sequential colimits. For each object $a \in \Cc$, let $Q_a$ be a set of sets of morphisms that are generic above $a$ with $|Q_a| \leq \lambda$. For any object $c \in \Cc$, there exists a sequential-colimit-preserving functor $F : \lambda \to \Cc$ such that $F(0) = c$ and for each $\alpha < \lambda$ and each $X \in Q_{F(\alpha)}$, there is a $\beta < \lambda$ with $\alpha < \beta$ such that $F(\alpha \to \beta) \in X$.
\end{prop}
\begin{proof}
  For each object $a \in \Cc$, let $(X_i^a : i < \lambda)$ be an enumeration of $Q_a$ (padded with instances of the full set of objects in $\Cc$ if $|Q_a| < \lambda$). Fix an enumeration $((\gamma_i,\delta_i) : i < \lambda)$ of $\lambda^2$ with the property that for every $(\alpha,\beta) \in \lambda^2$, the set $\{i < \lambda : (\gamma_i,\delta_i) = (\alpha,\beta)\}$ is cofinal in $\lambda$.

  Let $a_0=c$. At stage $i < \lambda$, given the object $a_i$, do the following:
  \begin{itemize}
  \item If $\gamma_i > i$, let $a_{i+1} = a_i$ and let $f_{i,i+1} : a_i \to a_{i+1}$ be the identity morphism.
  \item If $\gamma_i \leq i$, find some $f \in X_{\delta_i}^{a_{\gamma_i}}$ and let $f_{i,i+1} = f$ and $a_{i+1} = \cod(f)$. %
  \end{itemize}
  For each $j < i$, let $f_{j,i+1} : a_j \to a_{i+1}$ be $f_{i,i+1}\circ f_{j,i}$.

  For limit $i$, if $a_j$ is defined for all $j < i$ and $f_{j,k}$ is defined for all $j \leq k < i$, then let $a_i$ be the colimit of the $i$-indexed diagram given by $(a_j : j < i)$ and $(f_{j,k} : j\leq k < i)$. Let $f_{j,i} : a_j \to a_i$ be the corresponding canonical maps.

  Finally, let $F(i) = a_i$ for $i < \lambda$. For any $i \leq j < \lambda$, let $F(i \to j) = f_{i,j}$. This is a sequential-colimit-preserving functor by construction (since $F(i)$ is a colimit for each limit $i < \lambda$). We have that for any $\alpha < \lambda$ and $X \in Q_{F(\alpha)}$, there is a $\beta < \lambda$ such that $F(\beta) \in Q_{F(\alpha)}$ by our choice of the enumeration $((\gamma_i,\delta_i) : i < \lambda)$.
\end{proof}

\begin{defn}\label{defn:density-stuff}
  Given a $(k,1,1)$-weave model $(M,W,L,B,\prec,<,\eval,|\cdot|)$ and $\sigma \in W_{<L}$, a set $X \subseteq W_{\top}$ is \emph{dense above $\sigma$} if for every $\tau \in W_{<L}$ with $\sigma \prec \tau$, there is a $\gamma \in X$ with $\tau \prec \gamma$.

  A set $X \subseteq W_{\top}$ is \emph{somewhere dense} if it is dense above some $\sigma \in W_{<L}$.
\end{defn}

Note that $X$ being dense above $\sigma$ and being somewhere dense are both first-order definable in the structure $(M,W,L,\dots, X)$.

\begin{defn}
  An \emph{augmented $(k,1,1)$-weave model for $\varphi(x,y)$ (with index set $I$)} is a three-sorted structure $(M,W,L,B,\prec,<,\eval,|\cdot|,(P_i: i \in I))$ such that
  \begin{itemize}
  \item $(M,W,L,B,\prec,<,\eval,|\cdot|)$ is an unbounded $(k,1,1)$-weave model for $\varphi(x,y)$,
  \item for each $i \in I$, $P_i$ is a unary predicate selecting out subsets of $W_{\top}$, and
  \item for every finite $I_0 \subseteq I$, $\bigcap_{i \in I_0}P_i$ is somewhere dense.
  \end{itemize}
\end{defn}

Note that for a fixed $I$, the class of augmented $(k,1,1)$-weave models for $\varphi(x,y)$ with index set $I$ is axiomatizable in first-order logic.

We will often abbreviate $(M,W,L,B,\prec,<,\eval,|\cdot|,(P_i: i \in I))$ as $(M,W,L,(P_i: i \in I))$

\begin{defn}\label{defn:Weave-cat}
  Given augmented $(k,1,1)$-weave models $N = (M,W,L,B,\prec,<,\eval,|\cdot|,(P_i: i \in I))$ and $N' = (M',W',L',B',\prec',<',\eval',|\cdot|',(P_j: j \in J))$ a \emph{morphism from $N$ to $M$} is a pair $(f_0,f_1)$ where $f_1$ is an injection from $I$ into $J$ and $f_0$ is an elementary embedding of $(M,W,L,B,\prec,<,\eval,|\cdot|,(P_i: i \in I))$ into $(M',W',L',B',\prec',<',\eval',|\cdot|',(P_{g(i)}: i \in I))$. We will typically write $(f_0,f_1)$ as $f$. Composition of morphisms is componentwise composition: $f\circ g = (f_0 \circ g_0 ,f_1 \circ g_1)$. 

  We write $\Weav(T,k,\varphi)$ for the category of augmented $(k,1,1)$-weave models for $\varphi(x,y)$. Given a cardinal $\lambda$, let $\Weav_{\lambda}^0(T,k,\varphi)$ be the full subcategory of $\Weav(T,k,\varphi)$ consisting of augmented $(k,1,1)$-weave models $(M,W,L,(P_i: i \in I))$ with $|M|,|W|,|L|,|I| \leq \lambda$. In order to make this a small category, let $\Weav_\lambda(T,k,\varphi)$ be some small full subcategory of $\Weav^0_\lambda(T,k,\varphi)$ such that the inclusion functor of $\Weav_\lambda(T,k,\varphi)$ into $\Weav^0_\lambda(T,k,\varphi)$ is essentially surjective.\footnote{For a canonical choice, we can take $\Weav_\lambda(T,k,\varphi)$ to be the intersection of $\Weav^0_\lambda(T,k,\varphi)$ with the first level of the cumulative hierarchy $V_\alpha$ such that the inclusion functor is essentially surjective (i.e., the first level at which $\Weav^0_\lambda(T,k,\varphi) \cap V_\alpha$ contains every isomorphism type of $\Weav^0_\lambda(T,k,\varphi)$). Such an $\alpha$ always exists.}

  When $T$, $k$, and $\varphi$ are clear from context, we will write $\Weav$ and $\Weav_\lambda$ instead of $\Weav(T,k,\varphi)$ and $\Weav_\lambda(T,k,\varphi)$.
\end{defn}

It is straightforward to verify that $\Weav$ has arbitrary sequential colimits and $\Weav_\lambda$ has $\lambda$-sequential colimits.

For the remainder of this section, fix a theory $T$, $k< \omega$, and formula $\varphi(x,y)$ such that $T$ has $(k,1,1)$-weaves for $\varphi(x,y)$ of depth $\omega$. Fix also an infinite cardinal $\lambda \geq |T|$.

Given a formula $\psi = \psi(x,\cbar)$ in the language of $N = (M,W,L,(P_i: i \in I))$ with parameters $\cbar$ from $N$ and given a morphism $f: N \to N' = (M',W',L',(P_j: j \in J))$, let $\psi^{f}$ be the formula $\psi'(x,f_0(\cbar))$, where $\psi'$ is $\psi$ with each instance of $P_i$ replaced with $P_{f_1(i)}$. Finally, let an \emph{$N$-formula} be a formula in the language of $N$ with parameters from $N$.

\begin{lem}\label{lem:various-generic-sets}
  For any $(k,1,1)$-weave model $N = (M,W,L,(P_i: i \in I))$ in $\Weav_\lambda$, the following sets of morphisms are generic above $N$ in the category $\Weav_\lambda$ (where $N'$ is $(M',W',L',(P_j: i \in I'))$).
  \begin{enumerate}
  \item\label{generic-below} The set of morphisms $f : N \to N'$ such that for some $\sigma \in W'_{<L'}$ and $j \in I'$,
    \begin{itemize}
    \item $P_{g(i)}$ is dense above $\sigma$ for every $i \in I$,
    \item $N' \models P_j \subseteq P_{g(i)}$ for every $i \in I$, and
    \item every $\tau$ in $P_j$ extends $\sigma \concat (1,1)$.
    \end{itemize}
  \item\label{generic-force-heir-coheir} For any $N$-formula $\psi(x,\ybar)$ with $x$ a variable of sort $W$, the set of morphisms $f : N \to N'$ such that for some $j \in I'$, either
    \begin{itemize}
    \item there is a $\cbar \in N'^{\ybar}$ such that $P_j \subseteq \psi^{f,g}(N',\cbar)$ or
    \item for any morphism $f' : N' \to N''$ in $\Weav$ and any $\cbar \in N''$, $P_{g'(j)} \cap \psi^{f'\circ f}(N'',\cbar)$ is nowhere dense.
    \end{itemize}
  \end{enumerate}
\end{lem}
\begin{proof}
  For (\ref{generic-below}), fix a morphism $f^\ast : N \to N^\ast$ in $\Weav_\lambda$. Think of $N^\ast=(M^\ast,W^\ast,L^\ast,(P_i : i \in I^\ast))$ as a four-sorted structure with $I^\ast$ the fourth sort and $P$ coded as a binary relation. By the finite intersection condition on $P_i$, we can find a sufficiently saturated elementary extension $N^{\ast\ast} = (M^{\ast\ast},W^{\ast\ast},L^{\ast\ast},(P_i : i \in I^{\ast\ast}))$ (which is possibly larger than $\lambda$) and a $j \in I^{\ast\ast}$ such that $N^{\ast\ast} \models P_{j} \subseteq P_i$ for every $i \in I^\ast$. Since $P_j$ is somewhere dense, we can find a $\sigma \in W^{\ast\ast}_{<L^{\ast\ast}}$ such that $P_j$ is dense above $\sigma$. Fix an index element $\ell$ not in $I^{\ast\ast}$. Let $P_\ell$ be the set of elements of $P_j$ extending $\sigma \concat (1,1)$. Now consider the augmented $(k,1,1)$-weave model $N^{\dagger} = (M^{\ast\ast},W^{\ast\ast},L^{\ast\ast},(P_i : i \in I^{\ast}\cup\{\ell\}))$. This has a language of size at most $\lambda$, so by downward L\"{o}wenheim-Skolem, we can find an elementary submodel $N'$ of $N^\dagger$ containing $N^\ast$. Let $f' : N^\ast \to N'$ be the inclusion morphism. Then we now have that $f'\circ f^\ast$ is in the required set of morphisms. Since we can do this for any such $f^\ast : N \to N^\ast$, the set is dense above $N^\ast$.

  For (\ref{generic-force-heir-coheir}), fix a formula $\psi(x,\ybar)$ and a morphism $f^\ast : N \to N^\ast = (M^\ast,W^\ast,L^\ast,(P_i : i \in I^\ast))$ in $\Weav_\lambda$. If there exists a morphism $f' : N^\ast \to N'' = (M'',W'',L'',(P_i: i \in I''))$ in $\Weav$, $\cbar \in N''$, and $j \in I''$ such that $N'' \models P_j \subseteq \psi^{f'\circ f^\ast}(N'',\cbar)$, let $N'$ be an elementary substructure of $(M'',W'',L'',(P_i : i \in f_1'``[I^\ast] \cup \{j\})$ containing $\cbar$. Then $f'\circ f^\ast : N \to N'$ is the required morphism.

  Otherwise, if no such extension exists, find a morphism $f' : N^\ast \to N'$ as in the proof of (\ref{generic-below}). We need to argue that the new index $\ell$ is the required $j$ in the statement of the lemma. (The relevant fact is that $N' \models P_\ell \subseteq P_{f_1'(i)}$ for every $i \in I^\ast$.) Fix a morphism $f^\dagger : N' \to N^\dagger = (M^\dagger,W^\dagger,L^\dagger,(P_i : i \in I^\dagger))$. Assume for the sake of contradiction that $P_{f_1^\dagger(\ell)} \cap \psi^{f^\dagger \circ f' \circ f^\ast}(N^\dagger,\cbar)$ is somewhere dense for some $\cbar \in N^\dagger$. Fix some $\sigma \in W^{\dagger}_{<L^\dagger}$ such that $P_{f_1^\dagger(\ell)} \cap \psi^{f^\dagger \circ f' \circ f^\ast}(N^\dagger,\cbar)$ is dense above $\sigma$. Let $r$ be an index not in $I^\dagger$ and let $P_r = P_{f_1^\dagger(\ell)} \cap \psi^{f^\dagger \circ f' \circ f^\ast}(N^\dagger,\cbar)$. Consider the augmented $(k,1,1)$-weave model $(M^\dagger,W^\dagger,L^\dagger, (P_i : i \in f_1^\dagger ``[I'] \cup \{r\}))$. This now satisfies the extension condition, but we assumed that no such extension exists, which is a contradiction. Therefore $P_{f_1^\dagger(\ell)} \cap \psi^{f^\dagger \circ f' \circ f^\ast}(N^\dagger,\cbar)$ must be nowhere dense, as required.
\end{proof}

  Note that for each $N \in \Weav_\lambda$, there are at most $\lambda$ generic sets listed in (\ref{generic-below}) and (\ref{generic-force-heir-coheir}) in \cref{lem:various-generic-sets}.

The following lemma is essentially the same as \cite[Lem.~1.4]{NCTP}, but enough details of the formalism are different that we should state the result precisely and prove it again. We will take the opportunity to make it slightly more general as well.

Say that a filter $\Fc$ on a topological space $(X,\tau)$ is \emph{everywhere somewhere dense} if every $A \in \Fc$ is somewhere dense (i.e., satisfies that there is a non-empty open set $U \subseteq X$ such that $A$ is dense in $U$).

\begin{lem}\label{lem:everything-everywhere}
  Any everywhere somewhere dense filter $\Fc$ on a topological space $(X,\tau)$ can be extended to an everywhere somewhere dense ultrafilter $\Uc$.
\end{lem}
\begin{proof}
  Fix an everywhere somewhere dense filter $\Fc$ and a set $A \subseteq X$. We need to show that either $\Fc \cup \{A\}$ generates an everywhere somewhere dense filter or $\Fc \cup \{X \setminus A\}$ generates an everywhere somewhere dense filter.

  If $\Fc \cup \{X \setminus A\}$ generates an everywhere somewhere dense filter, then we are done, so assume that $\Fc \cup \{X \setminus A\}$ does not generate an everywhere somewhere dense filter. Assume for the sake of contradiction that $\Fc \cup \{A\}$ does not generate an everywhere somewhere dense filter. This implies that we can find $B,C \in \Fc$ such that $B \cap (X \setminus A)$ is nowhere dense and $C \cap A$ is nowhere dense. We may assume that $B = C$. Since $B \in \Fc$, there is a non-empty open set $U$ such that $B$ is dense in $U$. Since $B \cap (X \setminus A)$ is not dense in $U$, there is a non-empty open subset $V \subseteq U$ such that $B \cap (X \setminus A) \cap V$ is empty. Since $B \cap A$ is not dense in $V$, there is a non-empty open subset $W \subseteq V$ such that $B \cap A \cap V$ is empty. Together these imply that $B \cap W$ is empty, but this contradicts the fact that $B$ is dense in $U$ and therefore dense in $W$. Therefore it must be the case that $\Fc \cup \{A\}$ generates an everywhere somewhere dense filter.

  Since we can do this for any set $A \subseteq X$, we have by Zorn's lemma that we can extend $\Fc$ to an everywhere somewhere dense ultrafilter.
\end{proof}

\cite[Lem.~1.4]{NCTP} is the specific case of \cref{lem:everything-everywhere} applied to $2^{<\omega}$ with the topology generated by sets of the form $\{\tau \in 2^{<\omega} : \tau \succeq \sigma\}$ for $\sigma \in 2^{<\omega}$. In this paper, we will be applying \cref{lem:everything-everywhere} to $W_{\top}$ with the topology generated by sets of the form $\{\alpha \in W_{\top} : \alpha \succ \sigma\}$ for $\sigma \in W_{< L}$. Note that our previous use of the term dense is compatible with this topological interpretation. Specifically, $A \subseteq W_{\top}$ is dense above $\sigma$ if and only if it is topologically dense in the set $\{\alpha \in W_{\top} : \alpha \succ \sigma\}$ for $\sigma \in W_{< L}$ and $A$ is somewhere dense in the sense of \cref{defn:density-stuff} if and only if it is somewhere dense in the standard topological sense.

\begin{figure}
  \centering
  \begin{tikzpicture}[scale=0.75]
    \draw[dotted,line width=1.1pt] (0,0) -- (0,16) -- (16,16) -- (16,0) -- cycle;
    \draw[dotted,line width=1.1pt] (0,8) -- (16,8);
    \draw[dotted,line width=1.1pt] (8,0) -- (8,16);
    \foreach \i in {0,1} {
      \foreach \j in {0,1} {
        \begin{scope}[shift={(8*\i,8*\j)}]
          \draw[dotted,line width=1.1pt] (0,4) -- (8,4);
          \draw[dotted,line width=1.1pt] (4,0) -- (4,8);
          \foreach \ii in {0,1} {
            \foreach \jj in {0,1} {
              \begin{scope}[shift={(4*\ii,4*\jj)}]
                \draw[dotted,line width=1pt] (0,2) -- (4,2);
                \draw[dotted,line width=1pt] (2,0) -- (2,4);
                \foreach \iii in {0,1} {
                  \foreach \jjj in {0,1} {
                    \begin{scope}[shift={(2*\iii,2*\jjj)}]
                      \draw[dotted,line width=0.5pt] (0,1) -- (2,1);
                      \draw[dotted,line width=0.5pt] (1,0) -- (1,2);
                      \foreach \iiii in {0,1} {
                        \foreach \jjjj in {0,1} {
                          \begin{scope}[shift={(\iiii,\jjjj)}]
                            \draw[dotted,line width=0.25pt,opacity=0.75] (0,0.5) -- (1,0.5);
                            \draw[dotted,line width=0.25pt,opacity=0.75] (0.5,0) -- (0.5,1);
                            \foreach \iiiii in {0,1} {
                              \foreach \jjjjj in {0,1} {
                                \begin{scope}[shift={(0.5*\iiiii,0.5*\jjjjj)}]
                                  \draw[dotted,line width=0.1pt,opacity=0.75] (0,0.25) -- (0.5,0.25);
                                  \draw[dotted,line width=0.1pt,opacity=0.75] (0.25,0) -- (0.25,0.5);
                                \end{scope}
                              }
                            }
                          \end{scope}
                        }
                      }
                    \end{scope}
                  }
                }
              \end{scope}
            }
          }
        \end{scope}
      }
    }

    \newcommand\URbox[4]{
      \begin{scope}[scale={#1},shift={#2},transform shape]
        \draw[line width=#3] (0,2) -- (2,2) -- (2,1) -- (0,1) -- cycle;
        \draw[line width=#3] (2,0) -- (2,2) -- (1,2) -- (1,0) -- cycle;
        \node[scale=1.5] at (1.5,0.5) {$U_{#4}$};
        \node[scale=1.5] at (0.5,1.5) {$R_{#4}$};
      \end{scope}
    }
    \URbox{8}{(0,0)}{3pt}{0}
    \URbox{2}{(4,6)}{2pt}{1}
    \URbox{1}{(10,14)}{1pt}{2}
    \URbox{0.25}{(46,60)}{0.5pt}{3}
    \URbox{0.25/4}{(46*4 + 4,60*4 + 6)}{0.25pt}{4}
  \end{tikzpicture}
  \caption{The sets $U_j$ and $R_j$ in the proof of \cref{prop:get-heir-coheir-heir-coheir-from-weave}. Any sequence $(a_j : j < \cf(\lambda))$ with $a_j \in U_j$ for each $j < \cf(\lambda)$ is an up-$1$-comb, and any sequence $(b_j : j < \cf(\lambda))$ with $b_j \in R_j$ for each $j < \cf(\lambda)$ is a right-$1$-comb.}

  \label{fig:U-and-R}
\end{figure}
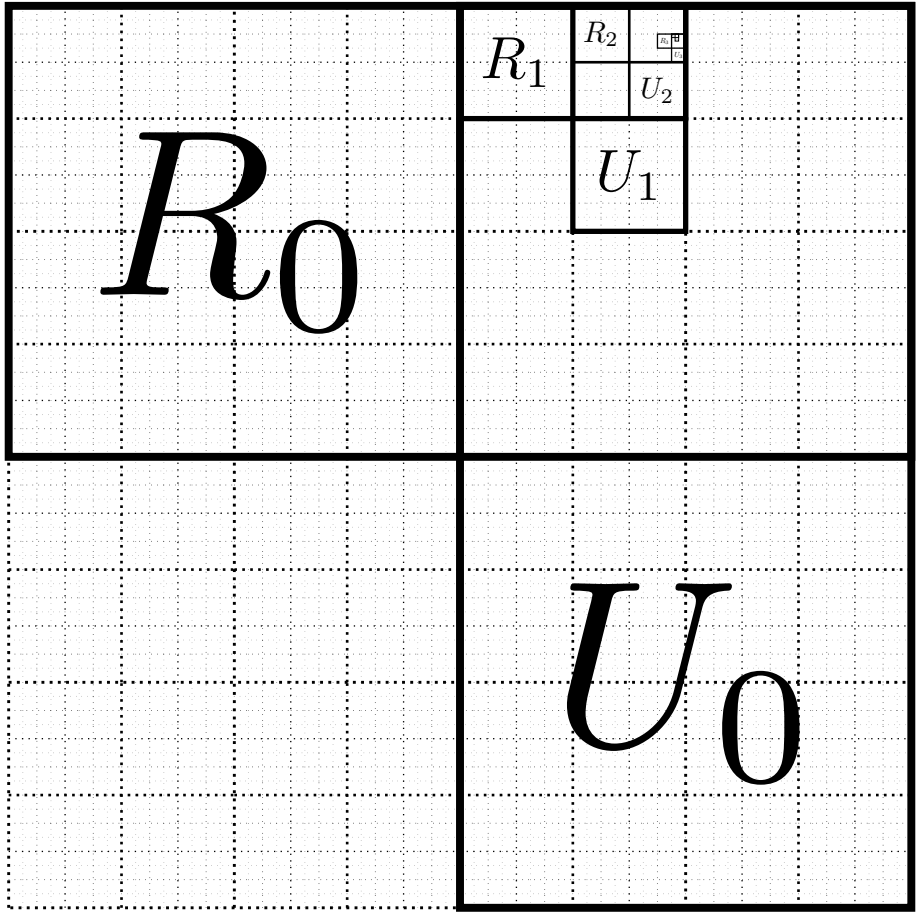

\begin{prop}\label{prop:get-heir-coheir-heir-coheir-from-weave}
  For any first-order theory $T$, formula $\varphi(x,y)$, and $k < \omega$, if $T$ has a $(k,1,1)$-weave for $\varphi(x,y)$ of depth $\omega$, then for any cardinal $\lambda \geq |T|$, there is a model $M$ with $|M| = \lambda$, a parameter $b$, and $M$-heir-coheirs $p(y),q(y) \supset \tp(b/M)$ such that $\varphi(x,b)$ $k$-divides along $p$ but does not divide along $q$.
\end{prop}
\begin{proof}
  For each $N = (M,W,L,(P_i : i \in I)) \in \Weav_\lambda$, let $(\ref{generic-below})_N$ be the set of morphisms in \cref{lem:various-generic-sets} (\ref{generic-below}) for the specific object $N$. For each $N$-formula $\psi(x,\ybar)$, let $(\ref{generic-force-heir-coheir})_{N,\psi}$ be the set of morphisms in \cref{lem:various-generic-sets} (\ref{generic-force-heir-coheir}) for the specific object $N$ and formula $\psi(x,\ybar)$. Let $Q_N = \{(\ref{generic-below})_N\} \cup \{(\ref{generic-force-heir-coheir})_{N,\psi} : \psi(x,\ybar)~\text{an}~N\text{-formula}\}$. 

  By \cref{prop:Baire-category-category} and \cref{lem:various-generic-sets}, we can build a sequential-colimit-preserving functor $F: \lambda \to \Weav_\lambda$ such that for every $\alpha < \lambda$ and $X \in Q_{F(\alpha)}$, there is a $\beta < \lambda$ with $\alpha < \beta$ such that $F(\alpha \to \beta) \in X$. Let $N_\lambda = (M_\lambda,W_\lambda,L_\lambda,(P_i : i \in I_\lambda))$ be the colimit in $\Weav$ of the diagram $F$. For each $\alpha < \lambda$, let $N_\alpha = (M_\alpha,W_\alpha,L_\alpha,(P_i : i \in I_\alpha))$ be the image of $F(\alpha)$ under the canonical morphism of $F(\alpha)$ into $N_\alpha$. Since this family is isomorphic to the image of $F(\alpha)$ (with each $F(\alpha \to \beta)$ taken to the inclusion map of $N_\alpha$ into $N_\beta$), we may assume that for each $\alpha < \lambda$, $F(\alpha)$ is $N_\alpha$ and for each $\alpha < \beta < \lambda$, $F(\alpha \to \beta)$ is the pair $(f,g)$, where $f$ is the inclusion map of $N_\alpha$ into $(M_\beta,W,\beta,L_\beta,(P_i : i \in I_\alpha))$ and $g$ is the inclusion map of $I_\alpha$ into $I_\beta$.

  Let $\Pc$ be the filter on $B``[W_\lambda] \subseteq M_\lambda$ generated by
  \begin{itemize}
  \item sets of the form $\{B(x) : N_\lambda \models P_i(x)\}$ and
  \item sets of the form $\{B(x) : x \in C\}$ for $C \subseteq (W_\lambda)_{\top}$ satisfying that $(W_\lambda)_{\top} \setminus C$ is nowhere dense.
  \end{itemize}
  Note that since $N_\lambda$ is an augmented $(k,1,1)$-weave model, $\Pc$ is everywhere somewhere dense (in the sense of the topology on $B``[W_\lambda]$ induced by sets of the form $\{B(x) : x \in (W_\lambda)_{\top},~x \succ b\}$ for $b \in (W_\lambda)_{<L_\lambda}$). In particular, this implies that $\Pc$ is non-trivial.

\vspace{0.5em}

\noindent \emph{Claim 1.} $\Pc$ generates a complete type over $N_\lambda$.

\begin{claimproof}{}
  Fix an $N_\lambda$-formula $\chi(x)$ with $x$ a variable in the $M_\lambda$ sort. Let $\psi(x) = \chi(B(x))$. Find an $\alpha < \lambda$ such that $\psi$ is an $N_\alpha$-formula. By construction, there is a $\beta < \lambda$ with $\alpha < \beta$ such that the inclusion morphism of $N_\alpha$ into $N_\beta$ is in $(\ref{generic-force-heir-coheir})_{N_\alpha,\psi}$. This implies that for some $j \in I_\beta$, either $N_\beta \models P_j \subseteq \psi(N_\beta)$, implying that $N_\lambda$ satisfies the same and thereby that $\psi(N_\lambda) \in \Pc$ or $P_j(N_\lambda) \cap \psi(N_\lambda)$ is nowhere dense in $(W_\lambda)_{\top}$. In the second case, we get immediately that the complement of $\psi(N_\lambda)$ is in $\Pc$. Since we can do this for any such formula, we have that $\Pc$ generates a complete type over $N_\lambda$.
\end{claimproof}

\vspace{0.5em}

\vspace{0.5em}

  \noindent \emph{Claim 2.} If $\Uc$ is an everywhere somewhere dense ultrafilter extending $\Pc$, then the $M_\lambda$-coheir (in the original language of $T$) generated by $\Uc$ is an $M_\lambda$-heir-coheir.

\begin{claimproof}{}
  Recall that $\Mb$ is the monster model of the theory $T$ (the theory of $M_\lambda$). Let $p(y)$ be the global type generated by $\Uc$. Fix some $M_\lambda$-formula $\psi(y,\zbar)$ and assume that for some $b \in \Mb$, $\psi(y,\bbar) \in p(y)$. We need to show that there is a $\cbar \in M_\lambda$ such that $\psi(y,\cbar) \in p(y)$ as well. Find $\alpha < \lambda$ such that $\psi(y,\zbar)$ is an $N_\alpha$-formula.

  Since $\Uc$ is everywhere somewhere dense, we must have that $\{a \in W_\lambda : \Mb \models \psi(B(a),\bbar)\}$ is somewhere dense. This means that when we met the condition $(\ref{generic-force-heir-coheir})_{N_\alpha,\psi(B(y),\zbar)}$, we must have satisfied the first bullet point, implying that there is a $\cbar \in N_\lambda$ such that $\psi(M_\lambda,\cbar)$ is an element of the filter $\Pc$ (and therefore of $\Uc$ as well). 
\end{claimproof}

\vspace{0.5em}

  Since we included the set $(\ref{generic-below})_N$ in $Q_N$, we can build an increasing cofinal sequence $(\alpha_j : j < \cf(\lambda))$ of ordinals less than $\lambda$, a sequence $(\sigma_j : j < \cf(\lambda))$ of elements of $(W_{\lambda})_{< L_\lambda}$, and a sequence $(i(j) : j < \cf(\lambda))$ of elements of $I_\lambda$ satisfying that for every $j < \cf(\lambda)$,
  \begin{itemize}
  \item $i(j) \in I_{\alpha_{j+1}}$ and $\sigma_j \in W_{\alpha_{j+1}}$,
  \item $P_i$ is dense above $\sigma_j$ for every $i \in I_{\alpha_j}$,
  \item $P_{i(j)} \subseteq P_i$ for every $i \in I_{\alpha_j}$, and
  \item every $\tau$ in $P_{i(j)}$ extends $\sigma_j\concat(1,1)$.
  \end{itemize}
  Let
  \begin{align*}
    U_j &= \left\{ B(x) : x \in (W_\lambda)_{\top},~ \sigma_j \concat (1,0) \prec x \right\},&   U &= \bigcup_{j < \cf(\lambda)}U_j, \\
    R_j &= \left\{ B(x) : x \in (W_\lambda)_{\top},~ \sigma_j \concat (0,1) \prec x \right\},&  R &= \bigcup_{j < \cf(\lambda)}R_j.
  \end{align*}

\vspace{0.5em}

\noindent \emph{Claim 3.} $\Pc \cup \{U\}$ and $\Pc \cup \{R\}$ both generate everywhere somewhere dense filters.

\begin{claimproof}{}
  To show that $\Pc \cup \{U\}$ generates an everywhere somewhere dense filter, it is sufficient to show that for any $C \in \Pc$, $U \cap C$ is somewhere dense. We may assume without loss of generality that $C = B``[P_i(N_\lambda) \setminus D]$ for some $i \in I_\lambda$ and some nowhere dense $D \subseteq (W_\lambda)_{\top}$. Find a $j$ such that $i \in I_{\alpha_j}$. By construction, we now have that $P_i(N_\lambda)$ is dense above $\sigma_j$, which implies that $P_i(N_\lambda)\setminus D$ is dense above $\sigma_j$ as well. This implies that $P_i(N_\lambda)\setminus D$ is dense above $\sigma_j \concat (1,0)$ and so $U \cap C$ is somewhere dense. 

  The proof for $\Pc \cup \{R\}$ is the same.
\end{claimproof}

\vspace{0.5em}

By \cref{lem:everything-everywhere}, we can find everywhere somewhere dense ultrafilters $\Uc_{\mathtt{U}}$ and $\Uc_{\mathtt{R}}$ extending $\Pc \cup \{U\}$ and $\Pc \cup \{R\}$, respectively. Let $p(y)$ be the global $N_\lambda$-coheir generated by $\Uc_{\mathtt{R}}$ and let $q(y)$ be the global $N_\lambda$-coheir generated by $\Uc_{\mathtt{U}}$. By Claim 1, we have that $p \res M_\lambda = q \res M_\lambda$. By Claim 2, the restrictions of $p(y)$ and $q(y)$ to the language of $T$ are $M_\lambda$-heir-coheirs. (This is true in the full language as well, but we will not need this.)

Now we just need to show that $\varphi(x,y)$ $k$-divides along $p(y)$ but does not divide along $q(y)$. This is easiest to see in the full language of $N_\lambda$. 

Let $(b_i)_{i < \omega}$ be a Morley sequence (in the monster model of $\Th(N_\lambda)$) generated by $p(y)$. For any finite $U_0 \subseteq U$, we have that for some $X \in \Uc_{\mathtt{U}}$, $U_0$ is narrowly below $X$. This implies the following statement by induction (on $n$):
\begin{itemize}
\item[$ $] For any $j_0 < j_1 <\dots j_{m-1} < \cf(\lambda)$ and any $a_0,\dots,a_{m-1} \in U$ with $a_\ell \in U_{j_\ell}$ for each $\ell < m$, the set $\{B^{-1}(a_0),\dots,B^{-1}(a_{m-1}), B^{-1}(b_{n-1}),B^{-1}(b_{n-2}),\dots,B^{-1}(b_0)\}$ is an up-$1$-comb, implying in particular that $\{\varphi(x,a_0),\dots,\varphi(x,a_{m-1}),\varphi(x,b_{n-1}),\dots,\varphi(x,b_0)\}$ is $k$-inconsistent.
\end{itemize}
This immediately implies that $\varphi(x,y)$ $k$-divides along $p(y)$.

The argument that $\varphi(x,y)$ does not divide along $q(y)$ is essentially the same.
\end{proof}

One thing to note is that like with the Baire category theorem, an advantage of the more abstract framework given by \cref{prop:Baire-category-category} is that we can easily ensure that the models built in \cref{prop:get-heir-coheir-heir-coheir-from-weave} satisfy other generic conditions without much extra work. For instance, if for a given cardinal $\kappa$, $\lambda$ satisfies that for any $\mu < \lambda$, $2^{|T| + \mu + \kappa} < \lambda$, then we can ensure that the model $M$ built in \cref{prop:get-heir-coheir-heir-coheir-from-weave} is $\kappa^+$-saturated. We can also simultaneously build many different pairs of heir-coheir $(p,q)$ satisfying the conclusion of \cref{prop:get-heir-coheir-heir-coheir-from-weave}, although it's unclear what this might be useful for.

\begin{thm}\label{thm:main-theorem}
  For any complete first-order theory $T$ and $k < \omega$, the following are equivalent.
  \begin{enumerate}
  \item\label{bi-rel-bi-rel} $T$ satisfies \textup{($k$, bi-invariant or semi-reliably invariant, bi-invariant or semi-reliably invariant)}--Kim's lemma.
  \item \label{bi-bi} $T$ satisfies $(k,$ \textup{bi-invariant, bi-invariant}$)$-\KL{}.
  \item \label{hc-hc} $T$ satisfies $(k,$ \textup{heir-coheir, heir-coheir}$)$-\KL{} over models.
  \item \label{weave-omega} $T$ does not have a $(k,1,1)$-weave of depth $\omega$.
  \item \label{strong-weave-omega} $T$ does not have a strong $(k,1,1)$-weave of depth $\omega$.
  \end{enumerate}
\end{thm}
\begin{proof}
  (\ref{bi-rel-bi-rel}) implies (\ref{bi-bi}) and (\ref{bi-bi}) implies (\ref{hc-hc}) by \cref{prop:Kim-monotone}. (\ref{weave-omega}) implies (\ref{bi-rel-bi-rel}) by \cref{cor:awkward}. (\ref{hc-hc}) implies (\ref{weave-omega}) by \cref{prop:get-heir-coheir-heir-coheir-from-weave}. Finally, (\ref{weave-omega}) and (\ref{strong-weave-omega}) are equivalent by \cref{prop:weave-iff-strong-weave}.
\end{proof}

(\ref{bi-rel-bi-rel}) in \cref{thm:main-theorem} is of course somewhat artificial, but it does have the advantage that it is both characterized by a forbidden combinatorial configuration and is non-trivial over arbitrary invariance bases, unlike the characterization of NCTP given in \cite{NCTP}.\footnote{Although it should be noted that \cite[Thm.~1.8, Prop.~2.6]{NCTP} and the fact that coheirs over models are extendibly invariant give a similarly artificial (or perhaps even more artificial) characterization of NCTP: A theory has $k$-CTP if and only if it fails to satisfy ($k$, extendibly invariant, bi-invariant or reliably invariant)--Kim's lemma. Like \cref{thm:main-theorem} (\ref{bi-rel-bi-rel}), this version of Kim's lemma is non-vacuous over any invariance base.\label{NCTP-footnote}}

Given the artificiality of \cref{thm:main-theorem} (\ref{bi-rel-bi-rel}) (and of the characterization of NCTP discussed in Footnote~\ref{NCTP-footnote}), the following question (which is similar in spirit to \cref{quest:bi-invariant-reliably-invariant}) seems reasonable.

\begin{quest}
  Is there a natural class $\Xc$ of invariant types mutually generalizing bi-invariant types and semi-reliably invariant types such that $(k,\Xc,\Xc)$--Kim's lemma is equivalent to the conditions in \cref{thm:main-theorem}? Is there a similar $\Yc$ such that NCTP is equivalent to ($k$, extendibly invariant, $\Yc$)--Kim's lemma?
\end{quest}

Using \cref{thm:main-theorem} we can of now show that Kim-forking with regards to these special invariant types entails Kim-dividing with regards to these special invariant types. This does require the use of reliably invariant types (rather than just semi-reliably invariant types), but these always exist over invariance bases by \cite[Thm.~2.14]{NCTP}.

\begin{cor}
  Fix a theory $T$ that does not have a $(k,1,1)$-weave of depth $\omega$ for any $k < \omega$. Fix also an invariance base $A$ and a formula $\varphi(x,b)$. Suppose that \(\varphi(x,b) \vdash \bigvee_{i < n} \psi_i(x,c_i) \) and for each $i < n$, $\psi_i(x,c_i)$ divides along an $A$-bi-invariant type or a semi-reliably $A$-invariant type. Then $\varphi(x,b)$ divides along a reliably $A$-invariant type.
\end{cor}
\begin{proof}
  (This proof is essentially identical to the proof of \cite[Cor.~2.16]{NCTP}.) Let $p(y,z_0,\dots,z_{n-1})$ be a reliably $A$-invariant type extending $\tp(bc_0\dots c_{n-1} / A)$ (which exists by \cite[Thm.~2.14]{NCTP}). Let $(d^je^j_i: i < n,~j<\omega)$ be a Morley sequence generated by $p$. Note that $(d^j : j < \omega)$ and $(e_i^j : j < \omega)$ for each $i < n$ are Morley sequences in reliably $A$-invariant types. By \cref{thm:main-theorem} and the fact that reliably $A$-invariant types are semi-reliably $A$-invariant, $\{\psi_i(x,e_i^j) : j < \omega\}$ is inconsistent for each $i < n$. By the standard argument, we have that $\{\bigvee_{i < n} \psi_i(x,e^j_i) : j < \omega\}$ is inconsistent. This implies that $\{\varphi(x,d^j) : j < \omega\}$ is inconsistent, whereby $\varphi(x,b)$ divides along $p \res y$. Therefore $\varphi(x,b)$ divides along a reliably $A$-invariant type, as restrictions of reliably invariant types to subtuples of variables are reliably invariant by definition.
\end{proof}

\section{\texorpdfstring{$(2,1,\omega)$}{(2,1,ω)}-weaves and cographs}
\label{sec:cographs}

In the specific case of the failure of $(2$, bi-invariant or semi-reliably invariant, strongly bi-invariant$)$-\KL{} (and therefore in the case of the failure of ($2$, $m$-strongly bi-invariant, strongly bi-invariant)--Kim's lemma for any positive $m \leq \omega$), we can give a far simpler description of the combinatorial configuration that arises.

\begin{defn}
  Fix two graphs $G_0 = (V_0,E_0)$ and $G_1 = (V_1,E_1)$ with $V_0 \cap V_1 = \varnothing$.
  \begin{itemize}
  \item The \emph{coproduct of $G_0$ and $G_1$} is the graph $(V_0\cup V_1,E_0\cup E_1)$.
  \item The \emph{graph join of $G_0$ and $G_1$} is the graph $(V_0 \cup V_1,E_0 \cup E_1 \cup \{\{a,b\} : a \in V_0,~b\in V_1\})$.
  \end{itemize}
  We also define coproducts and graph joins of graphs with not necessarily disjoint underlying sets in the obvious analogous way. We will denote the coproduct by $G_0 \oplus G_1$ and the graph join by $G_0 \graphjoin G_1$.

  The class of \emph{cographs} is the smallest class containing the singleton graph and closed under coproducts and graph joins.
\end{defn}

Cographs have the following forbidden subgraph characterization.

\begin{fact}[{\cite[Thm.~2]{Corneil_1981}}]\label{fact:cograph}
  The cographs are exactly the finite $P_4$-free graphs (i.e., graphs that do not have $P_4$ as an induced subgraph, where $P_4$ is the four-element path graph: $\bullet - \bullet - \bullet -\bullet$).
\end{fact}

\begin{defn}
  A theory $T$ \emph{admits arbitrary cograph consistency-inconsistency patterns for $\varphi(x,y)$} if for every cograph $(V,E)$, there is a family $(b_v : v \in V)$ of parameters such that for every $V_0 \subseteq V$, $\{\varphi(x,b_v) : v \in V_0\}$ is consistent if and only if $V_0$ is an anticlique.
\end{defn}

\begin{lem}\label{lem:finite-weaves-cographs}
  For any $d < \omega$, the graph $((2^2)^d,E_d)$ where $E_d = \{\{\sigma,\tau\} \subseteq (2^2)^d : \{\sigma,\tau\}~\text{is an up-}1\text{-comb}\}$ is a cograph.
\end{lem}
\begin{proof}
  Let $G_d = ((2^2)^d,E_d)$. $G_0$ is the singleton graph and so is obviously a cograph, and it is immediate that $G_{d+1}$ is isomorphic to $(G_d \graphjoin G_d) \oplus (G_d \graphjoin G_d)$.
\end{proof}

\begin{lem}\label{lem:unordered-pair-exact-char}
  For any $d < \omega$, any unordered pair $\{\sigma,\tau\} \subseteq (2^2)^d$ is exclusively either an up-$1$-comb or a wide right-$1$-comb.
\end{lem}
\begin{proof}
  Assume that $\{\sigma,\tau\}$ is not an up-$1$-comb. Let $\e$ be the greatest common initial segment of $\sigma$ and $\tau$. Since $\{\sigma,\tau\}$ is not an up-$1$-comb, it cannot be the case that $\{\sigma\}$ is narrowly above or below $\{\tau\}$. Therefore we must have that $\sigma$ extends $\e \concat (i,j)$ and $\tau$ extends $\e \concat (k,\ell)$ with $i \neq k$. Regardless of whether $i = 0$ or $i = 1$, this implies that $\{\sigma\}$ is widely to the left or widely to the right of $\{\tau\}$ and so $\{\sigma,\tau\}$ is a wide right-$1$-comb.
\end{proof}

\begin{lem}\label{lem:right-omega-comb-char}
  For any $d < \omega$, $C \subseteq (2^2)^d$ is a wide right-$\omega$-comb if and only if it does not have a subset that is an up-$1$-comb of size $2$.
\end{lem}
\begin{proof}
  First note that if $C \subseteq (2^2)^d$ is a wide right-$\omega$-comb, then every subset of it is as well, and so by \cref{lem:unordered-pair-exact-char}, we have that no subset of $C$ of size $2$ is an up-$1$-comb.

  To prove that if $C$ has no subset of size $2$ that is an up-$1$-comb, then $C$ is a wide right-$\omega$-comb, we will proceed by induction on the size of $C$. This is immediate for $|C|=1$ and for $|C|=2$, this follows from \cref{lem:unordered-pair-exact-char}. Now suppose that we know this for all $A \subseteq (2^2)^d$ with $|A| < n$ and fix some $C$ with $|C| = n$. 

  Let $\e$ be the greatest common initial segment of all elements of $C$. For each $(i,j) \in 2^2$, let $C_{i,j}$ be the set of elements of $C$ extending $\e \concat (i,j)$. By the choice of $\e$, it must be the case that at least two of the $C_{i,j}$'s are non-empty. If both $C_{0,0}$ and $C_{0,1}$ or if both $C_{1,0}$ and $C_{1,1}$ are non-empty, then $C$ has a subset of size $2$ that is an up-$1$-comb, which we have assumed does not happen. Therefore it must be the case that at most one of $C_{0,0}$ and $C_{0,1}$ and at most one of $C_{1,0}$ and $C_{1,1}$ is non-empty, which together with the induction hypothesis implies that $C$ is a wide right-$\omega$-comb.
\end{proof}

\begin{lem}\label{lem:embed-cograph-in-weave}
  For every cograph $(V,E)$, there is a $d < \omega$ and an injective function $f : V \to (2^2)^d$ such that for any $v_0,v_1 \in V$, $v_0 \mathrel{E} v_1$ if and only if $\{f(v_0),f(v_1)\}$ is an up-$1$-comb and $\neg v_0 \mathrel{E} v_1$ if and only if $\{f(v_0),f(v_1)\}$ is a wide right-$1$-comb.
\end{lem}
\begin{proof}
  Fix cographs $G_0 = (V_0,E_0)$ and $G_1 = (V_1,E_1)$ and suppose that we already have such functions $f_0 : V_0 \to (2^2)^{d_0}$ and $f_1 : V_1 \to (2^2)^{d_1}$. Let $d = \max\{d_0,d_1\}$. By embedding $(2^2)^{d_i}$ into $(2^2)^d$, we may assume that $d_0 = d_1 = d$.

  To build the required function $f$ for $G_0 \oplus G_1$, just take $f(v) = f_0(v) \concat (0,0)$ for $v \in V_0$ and $f(v) = f_1(v) \concat (1,0)$ for $v \in V_1$. And to build the required function $f$ for $G_0 \graphjoin G_1$, just take $f(v) = f_0(v) \concat (0,0)$ for $v \in V_0$ and $f(v) = f_1(v) \concat (0,1)$ for $v \in V_1$.

  By structural induction we are able to do this for all cographs.
\end{proof}

\begin{prop}\label{prop:weave-cograph-equiv}
  Fix a theory $T$ and a formula $\varphi(x,y)$. The following are equivalent.
  \begin{enumerate}
  \item\label{cograph-equiv-cograph} $T$ admits arbitrary cograph consistency-inconsistency patterns for $\varphi(x,y)$.
  \item\label{cograph-equiv-omega} $T$ has a (strong) $(2,\omega,\omega)$-weave for $\varphi(x,y)$ of depth $\omega$.
  \item\label{cograph-equiv-1} $T$ has a (strong) $(2,1,\omega)$-weave for $\varphi(x,y)$ of depth $\omega$.
  \item\label{cograph-equiv-d-omega} For every $d < \omega$, $T$ has a $(2,\omega,\omega)$-weave for $\varphi(x,y)$ of depth $d$.
  \item\label{cograph-equiv-d-1} For every $d < \omega$, $T$ has a $(2,1,\omega)$-weave for $\varphi(x,y)$ of depth $d$.
  \end{enumerate}
\end{prop}
\begin{proof}
  The strong and non-strong versions of (\ref{cograph-equiv-omega}) and (\ref{cograph-equiv-1}) are equivalent by \cref{prop:weave-iff-strong-weave}.
  
  The equivalence of (\ref{cograph-equiv-omega})-(\ref{cograph-equiv-d-1}) is immediate from \cref{prop:finite-depth-to-depth-omega} and the discussion after \cref{defn:weave}. 

  By \cref{lem:embed-cograph-in-weave}, we have that (\ref{cograph-equiv-omega}) implies (\ref{cograph-equiv-cograph}). By Lemmas~\ref{lem:finite-weaves-cographs} and \ref{lem:right-omega-comb-char}, we have that (\ref{cograph-equiv-cograph}) implies (\ref{cograph-equiv-d-omega}).
\end{proof}

We now get the following corollary, although it can also be proven fairly directly (in a manner analogous to the proof of \cref{thm:failure-of-Kim-bi-invariant-to-weaves}) without using the machinery of weaves.

\begin{cor}
  If $T$ fails to satisfy \textup{($2$, bi-invariant or semi-reliably invariant, strongly bi-invariant)}-\KL{}, then $T$ admits arbitrary cograph consistency-inconsistency patterns for some formula.
\end{cor}
\begin{proof}
  This follows immediately from \cref{cor:awkward} and \cref{prop:weave-cograph-equiv}. 
\end{proof}

\begin{quest}
  Is the failure of \textup{($2$, bi-invariant, strongly bi-invariant)}-\KL{} equivalent to admitting arbitrary cograph consistency-inconsistency patterns for some formula?
\end{quest}

The analogous question for semi-reliable invariance seems less tractable given that at the moment there is no known way to build a failure of a version of Kim's lemma for (semi-)reliably invariant types (or even just Kim-strictly invariant types) from a combinatorial configuration.

The following question is suggested by \cite[Lem.~3.20]{Ahn2022} together with the fact that $(k,\omega,\omega)$-weaves and $k$-ATP trees have a certain family resemblance.

\begin{quest}
  Are the equivalent conditions in \cref{prop:weave-cograph-equiv} equivalent to admitting a $(k,\omega,\omega)$-weave of depth $\omega$ for any $k < \omega$?
\end{quest}

\section{$k$-grids}
\label{sec:GSLC-failure}

In this section we will describe a forbidden combinatorial consistency-inconsistency configuration that is (modulo set-theoretic assumptions) an upper bound of two consequences of NATP considered in \cite{NCTP}, namely ($k$, invariant, strongly bi-invariant)--Kim's lemma and generic stationary local character.

To define generic stationary local character, we first need to recall that $[O]^\kappa$ is the set of subsets of $O$ of cardinality $\kappa$. A subset $C \subseteq [O]^\kappa$ is \emph{unbounded} if for every $X \in [O]^\kappa$, there is a $Y \in C$ with $X \subseteq Y$. $C \subseteq [O]^\kappa$ is \emph{closed} if for any increasing chain $(X_i : i < \alpha)$ (with $\alpha \leq \kappa$) with $X_i \in C$ for each $i < \alpha$, $\bigcup_{i < \alpha} X_i \in C$. $C \subseteq [O]^\kappa$ is a \emph{club} if it is closed and unbounded. $S \subseteq [O]^\kappa$ is \emph{stationary} if for every club $C \subseteq [O]^\kappa$, $S \cap C$ is non-empty.

\begin{defn}\label{defn:gslc}
  For any type $r(x)$ and any small $M,N \models T$ with $M \preceq N$, we write $\Xi(p,M,N)$ for the following condition:
  \begin{itemize}
  \item[$ $] For any $M$-formula $\varphi(x,y)$ and any $d$ such that $\varphi(x,d) \in p(x)$ and $d \indi_M N$, $\varphi(x,d)$ does not Kim-divide over $N$.
  \end{itemize}
  $T$ satisfies \emph{generic stationary local character} if for every $\lambda$, there is a $\kappa \geq \lambda$ such that for every $\kappa^+$-saturated model $O$, type $p \in S(O)$, and $M \preceq O$ with $|M| \leq \lambda$, $\{N \preceq O : N \succeq M,~|N| \leq \kappa,~\Xi(p,M,N)\}$ is stationary in $[O]^\kappa$.

\end{defn}

Whenever we talk about chains or antichains in $L^2$ (for a linear order $L$) it will be in the sense of the product partial order (i.e., $(i,j) \leq (k,\ell)$ if and only if $i \leq k$ and $j \leq \ell$). Recall that a \emph{strict chain} in a partial order is a set $C$ satisfying that for any distinct $x,y \in C$, either $x < y$ or $y < x$.

\begin{defn}\label{defn:k-grid}
  Given a linear order $L$ and $k < \omega$, a \emph{$k$-grid for $\varphi(x,y)$ indexed by $L$} is a family $(b_{i,j}:i,j \in L)$ of parameters in the sort of $y$ satisfying that
  \begin{itemize}
  \item for any strict chain $C \subseteq L^2$, $\{\varphi(x,b_{i,j}) : (i,j) \in C\}$ is consistent and
  \item for any antichain $A \subseteq L^2$, $\{\varphi(x,b_{i,j}) : (i,j) \in A\}$ is $k$-inconsistent.
  \end{itemize}
  An \emph{infinite $k$-grid} is a $k$-grid indexed by $L$ for some infinite $L$.
\end{defn}

An easy compactness argument gives that if a theory $T$ has an infinite $k$-grid for $\varphi(x,y)$, then it has a $k$-grid for $\varphi(x,y)$ indexed by $L'$ for every infinite $L'$. The analogous configuration with $k$-inconsistent strict chains and consistent antichains is clearly equivalent. It is possible to show directly that any theory with an infinite grid for some formula has ATP, but we will get this as a corollary of other results.

Just like with weaves and strong weaves, it is natural to wonder whether the definition of $k$-grid needs to be stated in terms of strict chains rather than arbitrary chains. Say that $(b_{i,j}: i,j \in L)$ is a \emph{strong $k$-grid} if it is a $k$-grid with the additional property that for any chain $C \subseteq L^2$, $\{\varphi(x,b_{i,j}) : (i,j) \in C\}$ is consistent.

\begin{prop}
  A theory $T$ has an infinite $k$-grid for $\varphi(x,y)$ if and only if it has an infinite strong $k$-grid for $\varphi(x,y)$.
\end{prop}
\begin{proof}
  Assume that $T$ has an infinite $k$-grid $(b_{i,j} : i,j\in L)$ for $\varphi(x,y)$. Let $M$ be a model of $T$ containing a $k$-grid $(b_{i,j} : i,j \in \Rb)$ indexed by $(\Rb, <)$. Consider $(M,\Rb,B)$ be a structure (including the original structure on $M$ and the field structure on $\Rb$) with $B : \Rb^2 \to M$ satisfying $B(i,j) = b_{i,j}$. Let $(N,K,B')$ be an elementary extension of $(M,\Rb,B)$ such that $K$ contains an infinitesimal $\e > 0$. For each $i,j \in \Rb$, let $c_{i,j} = B'((1-\e)i,(1-\e)j)$. It is now easy to verify that for any finite chain $C \subseteq \Rb^2$, $\{((1-\e)i,(1-\e)j):(i,j) \in C\}$ is a strict chain in $K^2$. Likewise for any finite antichain $A \subseteq \Rb^2$, $\{((1-\e)i,(1-\e)j):(i,j) \in A\}$ is an antichain in $K^2$. Therefore $(c_{i,j} : i,j \in \Rb)$ is an infinite strong $k$-grid for $\varphi(x,y)$.

  The other direction is immediate.
\end{proof}

By a similar argument, we of course also get that any theory with an infinite $k$-grid for a formula $\varphi(x,y)$ has, for any linear order $L$, an array $(b_{i,j} : i,j \in L)$ satisfying that $\{\varphi(x,b_{i,j}) : (i,j) \in C\}$ is $k$-inconsistent for any chain $C \subseteq L^2$ and $\{\varphi(x,b_{i,j}) : (i,j) \in A\}$ is consistent for any antichain $A$.\footnote{We did not discuss this earlier but a similar phenomenon happens with weaves where the presence of a $(k,m,n)$-weave for $\varphi(x,y)$ of depth $\omega$ implies (by an argument similar to the proof of \cref{prop:weave-iff-strong-weave}) that there is a family $(b_\sigma : \sigma \in (2^2)^\omega)$ such that for any up-$n$-comb $C \subseteq (2^2)^\omega$, $\{\varphi(x,b_\sigma) : \sigma \in C\}$ is consistent and for any wide right-$m$-comb $C \subseteq (2^2)^\omega$, $\{\varphi(x,b_\sigma) : \sigma \in C\}$ is consistent. It is unclear if this is really a meaningful observation but both of these configurations have the property that there are two equivalent stronger versions (one favoring the consistent sets of parameters and one favoring the $k$-inconsistent) which both break the symmetry between consistency and $k$-inconsistency.}

One thing to note is that it is relatively easy to embed a $(k,\omega,\omega)$-weave into a $k$-grid.

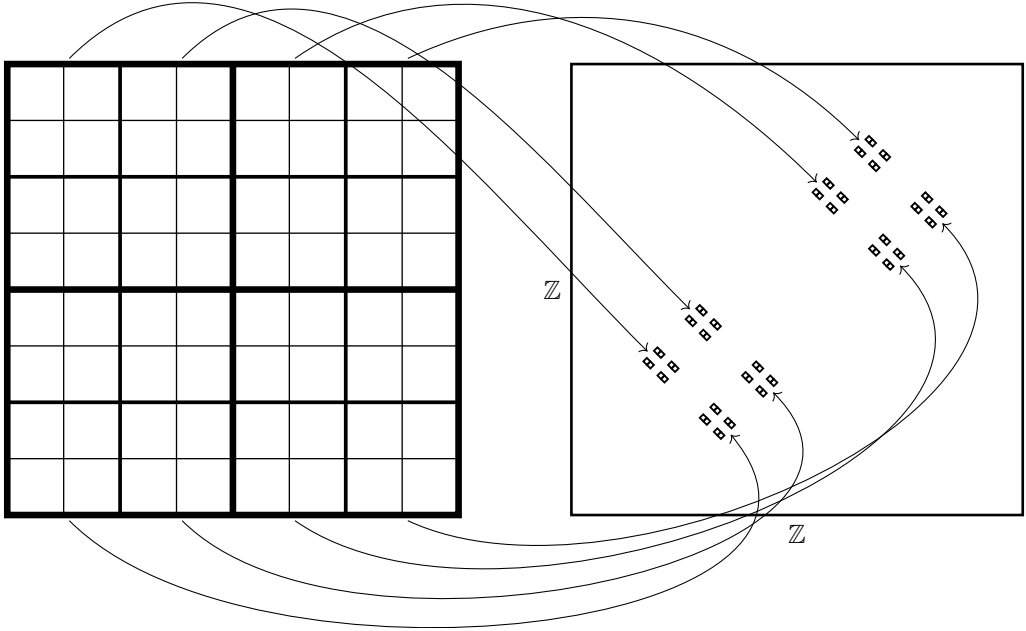
\begin{figure}
  \centering
  \begin{tikzpicture}
    \clip (-8,-1.5) rectangle (6.5,7);
    \begin{scope}[scale=0.75,shift={(-10,0)}]
      \draw[line width=2.5pt] (0,0) -- (0,8) -- (8,8) -- (8,0) -- cycle;
      \draw[line width=2.5pt] (0,4) -- (8,4);
      \draw[line width=2.5pt] (4,0) -- (4,8);
      \foreach \i in {0,1} {
        \foreach \j in {0,1} {
          \begin{scope}[shift={(4*\i,4*\j)}]
            \draw[line width=1.4pt] (0,2) -- (4,2);
            \draw[line width=1.4pt] (2,0) -- (2,4);
            \foreach \ii in {0,1} {
              \foreach \jj in {0,1} {
                \begin{scope}[shift={(2*\ii,2*\jj)}]
                  \draw[line width=0.5pt] (0,1) -- (2,1);
                  \draw[line width=0.5pt] (1,0) -- (1,2);
                \end{scope}
              }
            }
          \end{scope}
        }
      }
      \coordinate (A0000) at (1.1,-0.1);
      \coordinate (A0100) at (3.1,-0.1);
      \coordinate (A1000) at (5.1,-0.1);
      \coordinate (A1100) at (7.1,-0.1);
      \coordinate (B0011) at (1.1,8.1);
      \coordinate (B0111) at (3.1,8.1);
      \coordinate (B1011) at (5.1,8.1);
      \coordinate (B1111) at (7.1,8.1);
    \end{scope}

    \begin{scope}[scale=0.75]
      \draw[line width=1pt] (0,0) -- node[left] {$\Zb$} (0,8) -- (8,8) -- (8,0) -- node[below] {$\Zb$} cycle;
      \begin{scope}[shift={(2.65,1.35)},scale=1/sqrt(2),rotate=45]
        \foreach \i in {0,1} {
          \foreach \j in {0,1} {
            \begin{scope}[shift={(6*\i,2*\j)}]
              \foreach \ii in {0,1} {
                \foreach \jj in {0,1} {
                  \begin{scope}[shift={(1.5*\ii,0.5*\jj)}]
                    \coordinate (C\i\ii\j\jj) at (1/4,-0.1);
                    \coordinate (D\i\ii\j\jj) at (1/4,1/4+0.1);
                    \foreach \iii in {0,1} {
                      \foreach \jjj in {0,1} {
                        \begin{scope}[shift={(1.5/4*\iii,0.5/4*\jjj)}]
                          \draw[line width=0.7pt] (0,0) -- (0.125,0) -- (0.125,0.125) -- (0,0.125) -- cycle;
                        \end{scope}
                      }
                    }
                  \end{scope}
                }
              }
            \end{scope} 
          }
        }
      \end{scope}
    \end{scope}
    \draw[->] (B0011) to [out=45,in=135] (D0011);
    \draw[->] (B0111) to [out=45,in=135] (D0111);
    \draw[->] (B1011) to [out=35,in=135] (D1011);
    \draw[->] (B1111) to [out=25,in=135] (D1111);
    \draw[->] (A0000) to [out=-45,in=-50] (C0000);
    \draw[->] (A0100) to [out=-45,in=-45] (C0100);
    \draw[->] (A1000) to [out=-35,in=-45] (C1000);
    \draw[->] (A1100) to [out=-25,in=-45] (C1100);
  \end{tikzpicture}
  \caption{The function $f$ in \cref{prop:k-weave-in-k-grid}. The image of any right-$\omega$-comb is a strict chain and the image of any up-$\omega$-comb is an antichain.}
  \label{fig:weave-to-grid}
\end{figure}

\begin{prop}\label{prop:k-weave-in-k-grid}
  For any $d$, there is a map $f : (2^2)^d \to \Zb^2$ with the property that any up-$\omega$-comb $C \subseteq (2^2)^d$ is an antichain and any right-$\omega$-comb $C \subseteq (2^2)^d$ is a strict chain.
\end{prop}
\begin{proof}
  This is obvious for $d = 1$. Assume that we have such a map $f_d : (2^2)^d \to \Zb^2$ for some $d < \omega$. Find an $n$ large enough that for any $\sigma,\tau \in (2^2)^d$, $f_d(\sigma)$ and $f_d(\tau) +(n,-n)$  are incomparable in the product order on $\Zb$. Use this to extend $f_d$ to a map $f_{d+\nicefrac{1}{2}} : \{\sigma \concat (1,i) : \sigma \in (2^2)^d,~i < 2\}$ that still satisfies the required condition. Now find an $m$ large enough that for any $\sigma,\tau \in (2^2)$ and $i,j < 2$, $f_{d + \nicefrac{1}{2}}(\sigma \concat (1,i)) < f_{d + \nicefrac{1}{2}}(\tau \concat (1,j)) + (n,n)$. Use this to extend $f_{d + \nicefrac{1}{2}}$ to a function $f_{d+1}$ on all of $(2^2)^{d+1}$ satisfying the required condition.
\end{proof}

One might hope (as the author did) that the converse of \cref{prop:k-weave-in-k-grid} is also true, allowing us to show that the existence of $k$-grids is equivalent to the existence of $(k,\omega,\omega)$-weaves. Unfortunately this seems unlikely to follow directly. In the specific case of $2$-grids, we can see some evidence that this implication is unlikely by noting that the graph $(L^2, \{\{a,b\} \subseteq L^2 : a < b \vee b < a \})$ contains $P_4$ for any sufficiently large finite $L$ and so is not a cograph. This isn't conclusive of course. Given the general shape of these two conditions, namely the fact that weaves are naturally indexed by a pair of trees and grids are naturally indexed by a pair of sequences, the treelessness of \cite{Kaplan_2024} may be relevant.

\begin{quest}
  What is the relationship between theories that do not have an infinite $k$-grid for some $k < \omega$ and theories that do not admit $(k,\omega,\omega)$-weaves of depth $\omega$? Do treeless theories with $(k,\omega,\omega)$-weaves of depth $\omega$ always have infinite $k$-grids?
\end{quest}

For the remainder of the section, fix a first-order theory $T$ and $k < \omega$ such that $T$ has an infinite $k$-grid. Fix a $k$-grid $(b_{i,j} : i,j < \omega)$ for some formula $\varphi(x,y)$.

Fix a model $M \models T$ containing $(b_{i,j} : i,j < \omega)$ with $|M| \leq |T|$. Expand $M$ to the structure $(M,G, <_0,<_1,<,(b_{i,j} : i < \omega))$, where $G$ is a unary predicate selecting out the set $\{b_{i,j} : i,j < \omega\}$, $<_0$ and $<_1$ are the linear orders on the coordinates in $G$, $<$ is the product partial order on $G$, and $b_{i,j}$ is a constant for the element $b_{i,j}$. Let $T' = \Th(M,G, <_0,<_1,<,(b_{i,j} : i < \omega))$.

Given any $A \subseteq \omega^2$ and $n < \omega$, write $A(n)$ for the set $\{m < \omega : (n,m) \in A\}$. Let $\Fc$ be the filter on $\omega^2$ generated by sets $A \subseteq \omega^2$ satisfying that $\{n < \omega : A(n)~\text{is cofinite}\}$ is cofinite. Let $\Uc$ be an ultrafilter on $\omega^2$ extending $\Fc$. Let $q(y)$ be the global coheir generated by $\Uc$.

\begin{prop}\label{prop:q-doesn't-Kim-divide}
  $\varphi(x,y)$ does not divide along $q$.
\end{prop}
\begin{proof}
  $T'$ knows that for any finite strictly $<$-decreasing sequence $(c_i : i < n)$ in $G$, $\{\varphi(x,c_i) : i < n\}$ is consistent. Since $q(y)\vdash b_{i,j} < y$ for every $i,j < \omega$, we have that any Morley sequence generated by $q(y)$ is strictly $<$-decreasing. Therefore $\varphi(x,y)$ does not divide along $q$.
\end{proof}

Fix $\kappa \geq |T|$ and fix a $\kappa^+$-saturated model $O$ containing $(b_{i,j} : i,j < \omega)$. If $\kappa^+ = 2^\kappa$, choose $O$ so that $|O| = \kappa^+$. Let $\Elem_\kappa(O)$ be the set of elementary submodels of $O$ of size at most $\kappa$.

\begin{lem}\label{lem:build-sideways-elements}
  For any $N \in \Elem_\kappa(O)$, formula $\psi(y) \in q(y)$ with parameters in $N$, there is an elementary extension $N' \succeq N$ in $\Elem_\kappa(O)$ such that there is a $c \in G(N')\cap \psi(N')$ $b_{0,j} <_1 c$ for all $j < \omega$.
\end{lem}
\begin{proof}
  Since $\psi(y) \in q(y)$, we know that $G(y) \wedge \psi(y) \in q(y)$ as well. This implies that $\{(i,j) \in \omega^2 : N \models \psi(b_{i,j})\} \in \Uc$, so it must be the case that for infinitely many $i < \omega$, there are infinitely many $j < \omega$ such that $N^\ast \models \psi(b_{i,j})$. Fix some such $\ell'$ with $\ell' > \ell$. Let $X \subseteq \omega$ be the set of $j$ such that $N \models \psi(b_{\ell',j})$. By compactness and downward L\"{o}wenheim-Skolem, we can find an elementary extension $N' \succeq N$ in $\Elem_\kappa(O)$ such that there is a $c \in N' \setminus N$ with $\tp(c/N)$ finitely satisfiable in $\{b_{\ell,j} : j \in X\}$. It is immediate that $b_{0,j} <_1 c$ for all $j < \omega$.
\end{proof}

Several different sets of elementary extensions in $\Elem_\kappa(O)$ are generic (in the sense of \cref{defn:generic-above}, thinking of $\Elem_\kappa(O)$ as a posetal category):
\begin{itemize}
\item \cref{lem:build-sideways-elements} implies that for any $N \in \Elem_\kappa(O)$ and formula $\psi(y)$ with parameters in $N$, the set of elementary extensions $N' \succeq N$ satisfying the conclusion of \cref{lem:build-sideways-elements} is generic above $N$. Specifically, given an elementary extension $N^\ast \succeq N $ in $\Elem_\kappa(O)$ and a formula $\psi(y) \in q(y)$ with parameters in $N$, we can just apply \cref{lem:build-sideways-elements} to $N^\ast$ with the same choice of $\psi(y)$ to get an elementary extension $N' \succeq  N^\ast \succeq N$.
\item It is also easy to show that the set of elementary extensions $N' \succeq N$ in $\Elem_\kappa(O)$ satisfying that $q \res N$ is realized in $N'$ is generic above $N$.
\item For any $a \in O$, the set of elementary extensions $N' \succeq N$ in $\Elem_\kappa(O)$ with $a \in N'$ is generic above $N$ (although this will be too many requirements if $|O| > \kappa^+$).
\end{itemize}
By a standard argument (which could be thought of as an instance of \cref{prop:Baire-category-category}), we can build a continuous elementary chain $(N_i : i < \kappa^+)$ in $\Elem_\kappa(O)$ and a sequence $(e_i : i < \kappa^+)$ of elements of $O$ such that
\begin{itemize}
\item for every $i < \kappa^+$ and formula $\psi(y) \in q(y)$ with parameters in $N_i$, there is an $\alpha(i,\psi) < \kappa^+$ satisfying that for some $c_{\alpha(i,\psi)} \in G(N_{\alpha(i,\psi)})\cap \psi(N_{\alpha(i,\psi)})$, $b_{0,j} <_1 c_{\alpha(i,\psi)} <_1 c$ for all $j < \omega$,
\item for every $i < \kappa^+$, there is a $\e(i) < \kappa^+$ such that $e_i \in N_{\e(i)}$ and $e_i \models q \res N_i$, and
\item if $|O| = \kappa^+$, every $a \in O$ is an element of $N_i$ for some $i < \kappa^+$.
\end{itemize}
Let $C$ be the set of $\beta < \kappa^+$ with the property that
\begin{itemize}
\item for any $i < \beta$, $\psi(y) \in q(y)$ with parameters in $N_i$, $\alpha(i,\psi) < \beta$ and
\item for any $i < \beta$, $\e(i) < \beta$.
\end{itemize}
Since $\kappa^+$ is a regular cardinal and since $|N_i| \leq \kappa$ for all $i < \beta$, $C$ is a club in $\kappa^+$. Note that $(e_i : i \in C)$ is a Morley sequence generated by $q(y)$ and so $\{\varphi(x,e_i) : i \in C\}$ is consistent by \cref{prop:q-doesn't-Kim-divide}. Let $r(x)$ be some complete type over $O$ extending $\{\varphi(x,e_i) : i \in C\}$.

For each $\beta \in C$, let $H_\beta$ be the set of $c_{\alpha(i,\psi)}$ for $i < \beta$ and $\psi(y) \in q(y)$ with parameters in $N_i$. Note that by construction, $q \res N_\beta$ is finitely satisfiable in $H_\beta$ for any $\beta \in C$.

Let $\Fc_\beta$ be the filter on $H_\beta$ generated by the set of $d \in H_\beta$ such that $N_\beta \models \psi(d)$ for $\psi(y) \in q \res N_\beta$.

\begin{lem}\label{lem:down-right}
  Fix a $\beta \in C$. For every $c \in H_\beta$, there is an $A \in \Fc_\beta$ such that for any $d \in A$, $c <_0 d$ and $d <_1 c$.
\end{lem}
\begin{proof}
  Since $c \in H_\beta$, there is a $i < \omega$ such that $c <_0 b_{i,0}$. We have that $b_{0,0} <_0 y \wedge y <_1 c$ is  a formula in $q \res N_\beta$, so the required $A$ exists by construction.
\end{proof}

Let $\Vc_\beta$ be an ultrafilter on $H_\beta$ extending $\Fc_\beta$. Let $p_\beta(y)$ be a global coheir generated by $\Vc_\beta$. Note that by construction, we have that $q\res N_\beta = p_\beta \res N_\beta$.

\begin{prop}
  For each $\beta \in C$, $\varphi(x,y)$ $k$-divides along $p_\beta$.
\end{prop}
\begin{proof}
  \cref{lem:down-right} implies that for any $d \models p_\beta \res N_\beta$, we have that $c <_0 d$ and $d <_1 c$ for all $c \in H_\beta$. Therefore for any Morley sequence $(d_i : i < \omega)$ generated by $p_\beta$, we have that $d_{i+1} <_0 d_{i}$ and $d_i <_1 d_{i+1}$. $T'$ knows that $\varphi(x,y)$ is $k$-inconsistent on any set of elements of $G$ satisfying this condition, so $\varphi(x,y)$ $k$-divides along $p_\beta$.
\end{proof}

\begin{thm}\label{thm:grid-theorem}
  If $T$ has an infinite $k$-grid for $\varphi(x,y)$, then
  \begin{enumerate}
  \item\label{grid-c-shc} $T$ fails to satisfy \textup{($k$, coheir, strong heir-coheir)}--Kim's lemma over models,
  \item\label{grid-shc-c} $T$ fails to satisfy \textup{($k$, strong heir-coheir, coheir)}--Kim's lemma over models,
  \item\label{grid-weave} $T$ has a $(k,\omega,\omega)$-weave of depth $\omega$,
  \item\label{grid-hc-hc} $T$ fails to satisfy \textup{($k$, heir-coheir, heir-coheir)}--Kim's lemma over models, and
  \item\label{grid-gslc} if $\mathsf{GCH}$ holds, then $T$ fails to satisfy generic stationary local character.
  \end{enumerate}
\end{thm}
\begin{proof}
 For (\ref{grid-c-shc}), note that if $\kappa$ is sufficiently large and the model $N_0$ is chosen to be $|T|^+$-saturated, then $q(y)$ is a strong heir-coheir. In this case, $p_\beta(y)$ and $q(y)$ witness the failure of \textup{($k$, coheir, strong heir-coheir)}--Kim's lemma over models for any $\beta \in C$.

  (\ref{grid-weave}) and (\ref{grid-hc-hc}) follow from \cref{prop:finite-depth-to-depth-omega}, \cref{thm:main-theorem}, and \cref{prop:k-weave-in-k-grid}. 

  For (\ref{grid-gslc}), note that if $\mathsf{GCH}$ holds, then $\kappa^+ = 2^\kappa$ and so $\bigcup_{i < \kappa}N_i$ is all of $O$. Assume for the sake of contradiction that $X = \{N \preceq O : N \succeq N_0,~|N| \leq \kappa,~\Xi(r,N_0,N)\}$ is stationary in $[O]^\kappa$. $\{N_\beta : \beta \in C\}$ is a club in $[O]^\kappa$, so there is a $\beta \in C$ such that $N_\beta \in X$. We have that $\varphi(x,e_{\e(\beta)}) \in r(x)$. Since $e_{\e(\beta)} \models q \res N_\beta$, we have that $e_{\e(\beta)} \indi_{N_0} N_\beta$. But we also have that $\varphi(x,e_{\e(\beta)})$ $k$-divides along $p_\beta$, which contradicts the fact that $\Xi(r,N_0,N_\beta)$. Therefore it must not be the case that $X$ is stationary in $[O]^\kappa$. Since $\kappa$ was arbitrary, this implies that $T$ fails to satisfy generic stationary local character.

  (\ref{grid-shc-c}) follows by repeating the construction given in the section with the orientation of the grid rotated by $90\degree$ (so that in particular, $\varphi(x,b)$ $k$-divides along $q$ but does not divide along $p_\beta$ for any $\beta \in C$). 
\end{proof}

This could be shown in a direct combinatorial way in the same manner as \cref{prop:k-weave-in-k-grid}, but at this point we easily have the following corollary.

\begin{cor}
  If $T$ has an infinite $k$-grid, then $T$ has ATP.
\end{cor}
\begin{proof}
  This follows from \cref{thm:grid-theorem} (\ref{grid-c-shc}) and \cite[Prop.~1.7]{NCTP}.
\end{proof}

Infinite $k$-grids have a certain familial resemblance to ATP, but it is unclear whether they are equivalent.

\begin{quest}
  If $T$ has ATP, does it follow that $T$ has infinite $k$-grids?
\end{quest}

Although clearly there are weaker hypotheses than $\mathsf{GCH}$ that are sufficient for \cref{thm:grid-theorem} (\ref{grid-gslc}), the apparent need for some set-theoretic assumption may be a sign that  \cref{defn:gslc} is not a good definition. Regardless, the following question is natural.

\begin{quest}
  Is \cref{thm:grid-theorem} (\ref{grid-gslc}) provable in $\ZFC$?
\end{quest}

(\ref{grid-c-shc}) and (\ref{grid-hc-hc}) together also suggest the following question.

\begin{quest}
  If $T$ has an infinite $k$-grid, does it follow that $T$ fails to satisfy \textup{($k$, heir-coheir, strong heir-coheir)}--Kim's lemma over models? What about \textup{($k$, strong heir-coheir, heir-coheir)}--Kim's lemma or \textup{($k$, strong heir-coheir, strong heir-coheir)}--Kim's lemma over models?
\end{quest}

Given the simple form of \cref{defn:k-grid}, it seems plausible that one might be able to show that the existence of a $k$-grid in a theory $T$ entails the existence of a $2$-grid (analogously to how $k$-ATP implies $2$-ATP \cite[Lem.~3.20]{Ahn2022}).

\begin{quest}
  If a theory $T$ has a $k$-grid, does it follow that it also has a $2$-grid (possibly for a different formula)?
\end{quest}

What can be said is that both of these conditions imply NPM$^{(2)}$ in the sense of \cite[Def.~6.1]{Bailetti2024}. This follows from the relatively easy fact that PM$^{(2)}$ is equivalent to the existence of a consistency-inconsistency pattern indexed by the random graph.

\begin{defn}
  A theory $T$ admits a \emph{random graph consistency-inconsistency pattern (for $\varphi(x,y)$)} if there is a random graph $(V,E)$ and a family of parameters $(b_v : v \in V)$ such that for any $V_0 \subseteq V$, $\{\varphi(x,b_v) : v \in V_0\}$ is consistent if and only if $V_0$ is an anticlique.
\end{defn}

Clearly we have that a random graph consistency-inconsistency pattern for a formula $\varphi(x,y)$ entails both the existence of a $2$-grid for $\varphi(x,y)$ and the admission arbitrary cograph consistency-inconsistency patterns for $\varphi(x,y)$. It seems unlikely that converses of these implications are true, but at the moment no examples separating any of the conditions between NBTP and NPM$^{(2)}$ are known (see Figure~\ref{fig:implications}).

\begin{quest}
  If a theory $T$ has a $2$-grid or admits arbitrary cograph consistency-inconsistency patterns (for a single formula), does it follow that it admits a random graph consistency-inconsistency pattern?
\end{quest}

The only dead end in Figure~\ref{fig:implications} is ($k$, strongly bi-invariant, strongly bi-invariant)--Kim's lemma. Given the adjacent implications, the following question is reasonable.

\begin{quest}
  Does PM$^{(2)}$ entail the failure of \textup{($2$, strongly bi-invariant, strongly bi-invariant)}--Kim's lemma?
\end{quest}

Finally, it should be noted that NSOP$_4$ and binarity\footnote{A theory $T$ is \emph{binary} if every formula in $T$ is equivalent to a Boolean combination of formulas with two free variables.} do not entail the even weaker\footnote{Recall that definable types over models are strongly bi-invariant.} statement of ($2$, definable, definable)--Kim's lemma over models.

\begin{prop}
  For any non-empty set of parameters $A$ in the theory of the triangle-free random graph, there are two $A$-definable types $p(y_0,y_1)$ and $q(y_0,y_1)$ with $p \res A = q \res A$ such that $\varphi(x,y_0,y_1) = x \mathrel{E} y_0 \wedge x \mathrel{E} y_1$ $2$-divides with regards to $p$ but does not divide with regards to $q$.
\end{prop}
\begin{proof}
  Fix $a \in A$. Let $p(y_0,y_1)$ be the $A$-invariant type satisfying that
  \begin{itemize}
  \item $p(y_0,y_1) \vdash \neg y_0 \mathrel{E} y_1$,
  \item for any $b$, $p(y_0,y_1) \vdash y_0 \mathrel{E} b$ if and only if $b \mathrel{E} a$, and
  \item for any $b$, $p(y_0,y_1) \vdash y_1 \mathrel{E} b$ if and only if $b = a$.
  \end{itemize}
  Let $q(y_0,y_1)$ be the $A$-invariant type satisfying that
  \begin{itemize}
  \item $q(y_0,y_1) \vdash \neg y_0 \mathrel{E} y_1$,
  \item for any $b$, $p(y_0,y_1) \vdash \neg y_0 \mathrel{E} b$, and
  \item for any $b$, $p(y_0,y_1) \vdash y_1 \mathrel{E} b$ if and only if $b = a$.
  \end{itemize}
  These are both definable types. By quantifier elimination, it is immediate that $p \res A = q \res A$. A Morley sequence $(a_i^0,a_i^1 : i < \omega)$ generated by $p$ satisfies that for any $i < j < \omega$, $a_i^1 \mathrel{E} a_j^0$, and so $\{\varphi(x,a_i^0,a_i^1) : i < \omega\}$ is $2$-inconsistent. For any Morley sequence $(b_i^0,b_i^1 : i < \omega)$ generated by $q$, $\{b_i^0,b_i^1 : i < \omega\}$ is an anticlique, and so $\{\varphi(x,b_i^0,b_i^1) : i < \omega\}$ is consistent.
\end{proof}

This leads into one last question. It is relatively straightforward to show that \emph{any} first-order theory satisfies ($\omega$, arbitrary, generically stable)--Kim's lemma (where `arbitrary' means ordinary dividing, extending \cref{defn:Kim's-lemmas} in the obvious way to include dividing not necessarily along an invariant type). Examples of definable coheirs that are not finitely approximated (and therefore not generically stable) are hard to come by. The only known example, described in \cite[Sec.~7]{KeislerMeasuresWild}, lives in a fairly complicated theory $T_{\nicefrac{1}{2}}^\infty$, and characterizing dividing in this theory seems challenging.

\begin{quest}
  What can be said about variants of Kim's lemma involving finitely approximated types or definable coheirs?
\end{quest}

\begin{amssidewaysfigure}
  \centering

  \vspace{14.5em}
  
\[\begin{tikzcd}
	& {k\text{-NBTP}} & \begin{array}{c} \text{New Kim's Lemma} \\ \text{for }k\text{-dividing} \\ (k,\text{ inv., Kim-str.~inv.})\text{--KL} \end{array} & \\
	& {k\text{-NCTP}} & {(k,\text{ inv., bi-inv.})\text{--KL}} & \\
	{2\text{-NATP}} & {k\text{-NATP}} & & \\
	& {(k,\text{ inv., str.~bi-inv.})\text{--KL}} & {(k,\text{ bi-inv., bi-inv.})\text{--KL}} & \begin{array}{c} \text{No }(k,1,1)\text{-weave} \\ \text{of depth }\omega \end{array} \\
	\begin{array}{c} \text{Generic stationary} \\ \text{local character} \end{array} &  & {(k,\text{ bi-inv., str.~bi-inv.})\text{--KL}}  & \begin{array}{c} \text{No }(k,1,\omega)\text{-weave} \\ \text{of depth }\omega \end{array} \\
	&& {(k,\text{ str.~bi-inv., str.~bi-inv.})\text{--KL}} & \begin{array}{c} \text{No }(k,\omega,\omega)\text{-weave} \\ \text{of depth }\omega \end{array} \\
	& {\text{No infinite }k\text{-grid}} & {\hphantom{aaaaaaaaaaaaaaaaaaaaaaaaaaaaaaaaa}} \\
	&& {\text{NPM}^{(2)}} & \begin{array}{c} \text{Not all cograph} \\ \text{configurations} \end{array}
  \arrow["{\text{\cite[Thm.~5.2]{NKL}}}"', from=1-2, to=1-3]
  \arrow[from=1-2, to=2-2]
  \arrow[from=1-3, to=2-3]
	\arrow["{\text{\cite[Thm.~1.8]{NCTP}}}"', tail reversed, from=2-2, to=2-3]
	\arrow[from=2-2, to=3-2]
	\arrow[from=2-3, to=4-3]
	\arrow[from=4-2, to=5-3]
	\arrow["{\text{\cite[Lem.~3.20]{Ahn2022}}}"', tail reversed, from=3-1, to=3-2]
	\arrow["{\text{\cite[Prop.~1.7]{NCTP}}}", from=3-2, to=4-2]
	\arrow["{\text{\cite[Prop.~3.10]{NCTP}}}"', from=4-2, to=5-1] %
	\arrow["{\text{Thm.~\ref{thm:main-theorem}}}"', tail reversed, from=4-3, to=4-4]
	\arrow[from=4-3, to=5-3]
	\arrow[from=4-4, to=5-4]
	\arrow["\begin{array}{c} \text{Thm.~\ref{thm:grid-theorem}} \\ (\mathsf{GCH}) \end{array}"', dashed, from=5-1, to=7-2]
	\arrow["{\text{Thm.~\ref{thm:grid-theorem}}}", from=4-2, to=7-2]
	\arrow[from=5-3, to=6-3]
	\arrow["{\text{Thm.~\ref{thm:failure-of-Kim-bi-invariant-to-weaves}}}", from=5-4, to=5-3]
	\arrow[from=5-4, to=6-4]
	\arrow["{\text{Thm.~\ref{thm:failure-of-Kim-bi-invariant-to-weaves}}}", from=6-4, to=6-3]
	\arrow["{\text{Prop.~\ref{prop:k-weave-in-k-grid}}}", curve={height=-15pt}, from=6-4, to=7-2]
  \arrow[shift right = 3, "\begin{array}{c} (k=2) \end{array}"', dashed, from=6-4, to=5-4]
	\arrow["{\text{Prop.~\ref{prop:weave-cograph-equiv}}}"', from=6-4, to=8-4]
  \arrow[shift right = 3, "\begin{array}{c} \text{Prop.~\ref{prop:weave-cograph-equiv}} \\ (k=2) \end{array}"', dashed, from=8-4, to=6-4]
	\arrow["{}"', from=7-2, to=8-3] %
	\arrow["{}", from=8-4, to=8-3] %
\end{tikzcd}\]
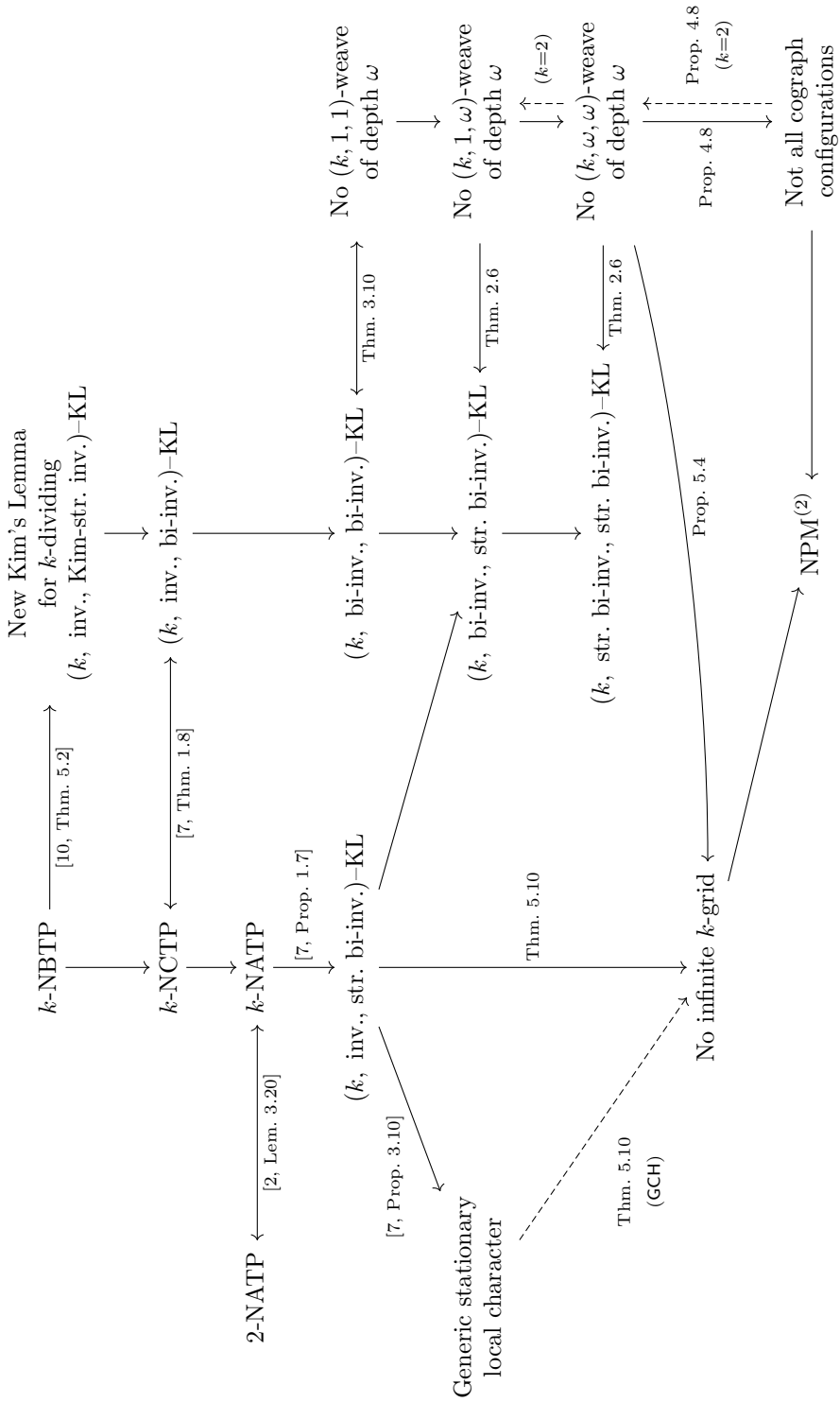
\caption[caption]{Diagram of some known implications (for a fixed $k$).}
  \label{fig:implications}
\end{amssidewaysfigure}

\newpage

\bibliographystyle{plain}
\bibliography{../ref}

\end{document}

%% file: macros.tex
\newtheorem{thm}{Theorem}[section]

\newtheorem{prop}[thm]{Proposition}
\newtheorem{lem}[thm]{Lemma}
\newtheorem{cor}[thm]{Corollary}
\newtheorem{fact}[thm]{Fact}

\newtheorem{quest}[thm]{Question}

\theoremstyle{definition}
\newtheorem{defn}[thm]{Definition}

\theoremstyle{remark}

\newenvironment{claimproof}[1]{\par\noindent\emph{Proof of claim.}\space#1}{\hfill $\vartriangleleft$}

\newcommand{\Rb}{\mathbb{R}}

\newcommand{\Zb}{\mathbb{Z}}

\newcommand{\Fc}{\mathcal{F}}
\newcommand{\Uc}{\mathcal{U}}
\newcommand{\Rc}{\mathcal{R}}

\newcommand{\Vc}{\mathcal{V}}

\newcommand{\Ic}{\mathcal{I}}

\newcommand{\Pc}{\mathcal{P}}

\newcommand{\Cc}{\mathcal{C}}

\newcommand{\Xc}{\mathcal{X}}
\newcommand{\Yc}{\mathcal{Y}}
\newcommand{\Mb}{\mathbb{M}}
\newcommand{\Zc}{\mathcal{Z}}

\makeatletter \DeclareRobustCommand{\cset}{\@ifstar\star@cset\normal@cset}
\newcommand{\star@cset}[1]{{\left\llbracket#1\right\rrbracket}}
\newcommand{\normal@cset}[2][]{{\mathopen{#1\llbracket}#2\mathclose{#1\rrbracket}}}
\makeatother

\newcommand{\xbar}{\bar{x}}
\newcommand{\ybar}{\bar{y}}
\newcommand{\zbar}{\bar{z}}

\newcommand{\bbar}{\bar{b}}
\newcommand{\cbar}{\bar{c}}

\newcommand{\ebar}{\bar{e}}

\DeclareMathOperator{\dom}{dom}

\newcommand{\IZF}{\ifmmode\mathsf{IZF}\else$\mathsf{IZF}$\fi}
\newcommand{\CZF}{\ifmmode\mathsf{CZF}\else$\mathsf{CZF}$\fi}
\newcommand{\ZF}{\ifmmode\mathsf{ZF}\else$\mathsf{ZF}$\fi}
\newcommand{\ZFC}{\ifmmode\mathsf{ZFC}\else$\mathsf{ZFC}$\fi}
\newcommand{\AC}{\ifmmode\mathsf{AC}\else$\mathsf{AC}$\fi}
\newcommand{\AD}{\ifmmode\mathsf{AC}\else$\mathsf{AD}$\fi}
\newcommand{\BZ}{\ifmmode\mathsf{BZ}\else$\mathsf{BZ}$\fi}
\newcommand{\GCH}{\ifmmode\mathsf{GCH}\else$\mathsf{GCH}$\fi}
\newcommand{\PA}{\ifmmode\mathsf{PA}\else$\mathsf{PA}$\fi}
\newcommand{\DC}{\ifmmode\mathsf{DC}\else$\mathsf{DC}$\fi}
\newcommand{\MP}{\ifmmode\mathsf{MP}\else$\mathsf{MP}$\fi}
\newcommand{\CT}{\ifmmode\mathsf{CT}\else$\mathsf{CT}$\fi}
\newcommand{\PAx}{\ifmmode\mathsf{PAx}\else$\mathsf{PAx}$\fi}
\newcommand{\VL}{\ifmmode{\mathsf{V}=\mathsf{L}}\else$\mathsf{V}=\mathsf{L}$\fi}

\DeclareMathOperator{\Th}{Th}

\newcommand{\res}{{\upharpoonright}}

\newcommand{\e}{\varepsilon}

\newlength{\savedparindent}
\AtBeginDocument{\setlength{\savedparindent}{\parindent}}

\DeclareMathOperator{\tp}{tp}
\DeclareMathOperator{\cf}{cf}

\def\Ind{\setbox0=\hbox{$x$}\kern\wd0\hbox to 0pt{\hss$\mid$\hss}
\lower.9\ht0\hbox to 0pt{\hss$\smile$\hss}\kern\wd0}

\def\Notind{\setbox0=\hbox{$x$}\kern\wd0\hbox to 0pt{\mathchardef
\nn=12854\hss$\nn$\kern1.4\wd0\hss}\hbox to
0pt{\hss$\mid$\hss}\lower.9\ht0 \hbox to 0pt{\hss$\smile$\hss}\kern\wd0}

\def\ind{\mathop{\mathpalette\Ind{}}}

\newcommand{\indf}{\ind^{\!\!\textnormal{f}}}

\newcommand{\indi}{\ind^{\!\!\textnormal{i}}}

\newcommand{\indu}{\ind^{\!\!\textnormal{u}}}

\newcommand{\indK}{\ind^{\!\!\textnormal{K}}}

\newcommand{\concat}{{\frown}}%

\newcommand{\cod}{\mathrm{cod}}

\newcommand{\fst}{1^{\text{st}}}

%% file: Kims-lemma-for-bi-invariant-types.bbl
\begin{thebibliography}{10}

\bibitem{ATP-1}
JinHoo Ahn and Joonhee Kim.
\newblock {SOP}$_1$, {SOP}$_2$, and antichain tree property.
\newblock {\em Annals of Pure and Applied Logic}, 175(3):103402, March 2024.

\bibitem{Ahn2022}
JinHoo Ahn, Joonhee Kim, and Junguk Lee.
\newblock On the antichain tree property.
\newblock {\em Journal of Mathematical Logic}, 23(02), December 2022.

\bibitem{Bailetti2024}
Michele {Bailetti}.
\newblock {A Walk on the Wild Side: Notions of maximality in first-order
  theories}.
\newblock {\em arXiv e-prints}, page arXiv:2409.19236, September 2024.

\bibitem{Conant2020}
Gabriel Conant and Kyle Gannon.
\newblock Remarks on generic stability in independent theories.
\newblock {\em Annals of Pure and Applied Logic}, 171(2):102736, February 2020.

\bibitem{KeislerMeasuresWild}
Gabriel Conant, Kyle Gannon, and James Hanson.
\newblock Keisler measures in the wild.
\newblock {\em Model Theory}, 2(1):1--67, June 2023.

\bibitem{Corneil_1981}
D.~G. Corneil, H.~Lerchs, and L.~Stewart Burlingham.
\newblock Complement reducible graphs.
\newblock {\em Discrete Applied Mathematics}, 3(3):163–174, July 1981.

\bibitem{NCTP}
James~E. {Hanson}.
\newblock {Bi-invariant types, reliably invariant types, and the comb tree
  property}.
\newblock {\em arXiv e-prints}, page arXiv:2306.08239, June 2023.

\bibitem{Kaplan_2024}
Itay Kaplan, Nicholas Ramsey, and Pierre Simon.
\newblock Generic stability independence and treeless theories.
\newblock {\em Forum of Mathematics, Sigma}, 12, 2024.

\bibitem{Some-Remarks-Kim-dividing-NATP}
Joonhee {Kim} and Hyoyoon {Lee}.
\newblock {Some Remarks on Kim-dividing in NATP Theories}.
\newblock {\em arXiv e-prints}, page arXiv:2211.04213, November 2022.

\bibitem{NKL}
Alex Kruckman and Nicholas Ramsey.
\newblock A {N}ew {K}im’s {L}emma.
\newblock {\em Model Theory}, 3(3):825–860, Aug 2024.

\bibitem{Mutchnik-NSOP2}
Scott {Mutchnik}.
\newblock {On NSOP$_2$ Theories}.
\newblock {\em arXiv e-prints}, page arXiv:2206.08512, June 2022.

\end{thebibliography}
